\documentclass[pdflatex,sn-mathphys-num]{sn-jnl}

\usepackage[true]{anonymous-acm}

\usepackage[table,dvipsnames,svgnames,usenames]{xcolor}
\usepackage{latexsym, amsthm, amsmath, bbm}
\usepackage{thmtools}
\usepackage{thm-restate} 
\usepackage{amssymb}
\usepackage{amsfonts}
\usepackage{amstext}
\usepackage{multicol}
\usepackage{graphicx}
\usepackage{grffile}

\usepackage{subcaption}
\usepackage{hyperref}

\usepackage{multirow}
\usepackage[small]{complexity}

\usepackage[utf8]{inputenc}

\usepackage{amsthm} 
\usepackage{amsmath}  

\usepackage[noabbrev,nameinlink,capitalize]{cleveref}

\usepackage{appendix}
\usepackage{url}

\usepackage{setspace}
\usepackage[numbers]{natbib}

\usepackage{graphicx}
\graphicspath{ {./images/} }

\usepackage{tikz}
\usetikzlibrary{calc,hobby,backgrounds,shapes,arrows,quotes,shapes.geometric,positioning,patterns,external,decorations.pathreplacing,calligraphy}
\usepackage{tikz-qtree}

\definecolor{aqua}{RGB}{0, 255, 255}
\definecolor{pink}{RGB}{255, 0, 255}

\usepackage{subcaption}

\usepackage{afterpage}

\usepackage{xspace}

\newcommand{\Z}{{\mathbb{Z}}}

\newcommand{\T}{{\mathcal{T}}}
\newcommand{\I}{{\mathcal{I}}}

\newcommand{\LL}{\mathrm{L}}
\newcommand{\RR}{\mathrm{R}}

\newcommand{\ctriangle}{even triangle\xspace}
\newcommand{\ctriangles}{even triangles\xspace}
\newcommand{\Ctriangle}{Even triangle\xspace}

\newtheorem{Theorem}{Theorem}
\newtheorem{Lemma}{Lemma}

\newtheorem{Definition}{Definition}

\theoremstyle{remark}

\usepackage{algorithm}
\usepackage{algpseudocodex}
\newcommand{\Input}[1]{%
  \Statex \hspace{\algorithmicindent}\textbf{Input:} #1%
}

\usepackage{bm}
\usepackage{bbm}

\begin{document}

\title{Connectivity of Districting Metagraphs}
\author[1]{\fnm{Mehmet} \sur{Emre}}\email{memre@usfca.edu}
\equalcont{Authors listed alphabetically, per convention.}

\author*[2]{\fnm{Daniel C.} \sur{Jerison}}\email{dcjerison@usfca.edu}
\equalcont{Authors listed alphabetically, per convention.}   

\author[1]{\fnm{Ellen} \sur{Veomett}}\email{eveomett@usfca.edu}
\equalcont{Authors listed alphabetically, per convention.}

\affil[1]{\orgdiv{Department of Computer Science}, \orgname{University of San Francisco}}

\affil[2]{\orgdiv{Department of Mathematics and Statistics}, \orgname{University of San Francisco}}

\begin{titlepage}
    \maketitle 
    \vspace{2cm}

\bmhead{Author Contribution}

All three authors contributed equally to the paper.

\bmhead{Funding Declaration}

No funding was received to assist with the preparation of this manuscript.

\bmhead{Competing Interests}

The authors declare that they have no competing interests.

\bmhead{Code Availability}

We have made all code used to run the computer search publicly available at the following GitHub repository: \url{https://github.com/maemre/triforce}
    
    \vfill 
\end{titlepage}

\begin{abstract}

In this article, we prove irreducibility results for a family of Markov chains arising in the study of redistricting and detecting gerrymandering. These chains use ReCom moves as their transition mechanism and are commonly employed in Markov chain Monte Carlo methods to generate ensembles of districting plans. Such ensembles are frequently used for outlier analysis, in which a proposed districting map is compared against the ensemble to determine whether it behaves atypically; this methodology often appears in expert testimony in redistricting litigation.

We show that when the underlying dual graph is a triangular subset of the triangular lattice with side length 5 or larger, and each district consists of two merged geographic regions, the associated ReCom chain is irreducible. This provides another entry in the very small list of known classes of ReCom chains for which irreducibility has been established.

We also demonstrate the fragility of this phenomenon by constructing an infinite family of maps for which the corresponding ReCom chain is not irreducible. Indeed, we produce a districting map that, after implementing a single ReCom move, always yields the same original map.  These examples remain structurally close to the triangular lattice: they arise as subdivisions of the triangular lattice, and the resulting graphs have maximum degree at most 8.

Finally, we prove irreducibility for a further special case: the ReCom chain on a $3 \times n$ grid graph partitioned into three districts of size $n$.

\end{abstract}


\section{Introduction}

In recent years, the study of redistricting has captured the interest of many mathematicians, computer scientists, and other computationally-minded members of the academy \cite{mgggWebpage, quantifyingGerrymanderingWebpage, PGP}.  In particular, many academics have used Markov chain Monte Carlo (MCMC) methods to construct ensembles of potential redistricting maps, against which a proposed map is compared \cite{ColoradoInContext, DukeNC, freeElectionsFreeState}.  With this ensemble, one can conduct an outlier analysis:  if the proposed map appears to be an outlier compared to the maps in the ensemble, this suggests the proposed map may have been drawn with partisan intent.

Among the methods used to construct this ensemble of maps, using a ReCom step in MCMC has emerged as a particularly popular method \cite{RecomMGGG}.  This method starts with the dual graph of the regions which are pieced together to comprise the districts (precincts, census blocks, etc).  Two adjacent districts are chosen, and we consider the subgraph induced by the nodes in those two districts.  A spanning tree is chosen within that subgraph, and then an edge is chosen to be cut so that the two remaining connected components have the correct population for a single district.  Those two new districts, along with the other unchanged districts, form a new districting map which is one step away from the original map in the Markov chain.  In this way, we consider what many have called the \emph{metagraph} of this Markov chain: each node in the metagraph is a districting map, and two maps are adjacent if they can be constructed from one another by a single 
ReCom step (see, for example, \cite{MGGGmetagraph}).

An important question whenever a finite Markov chain is utilized is whether that chain is \emph{irreducible}: whether all states can be reached from all other states in a finite number of steps.  In this context, irreducibility means that any valid map can be produced by a sequence of ReCom moves from a single starting map; equivalently, irreducibility means that the corresponding metagraph is connected.  The answer to this irreducibility question is generally unknown in the context of redistricting, when graphs come from real maps and districts must satisfy a population constraint.  Thus, several researchers have studied the connectivity of the ReCom metagraph in specific contexts.

There are many different variations of the metagraph connectivity question, depending on the choices made in the categories below:
\begin{itemize}
\item \textbf{Internal structure of the dual graph.} The underlying dual graph could be a square lattice \cite{JTF_tilings, Akitaya_et_al}, triangular lattice \cite{cannon_irreducibility, Akitaya_et_al}, or some other graph \cite{roising2023ergodic,nakano2010local}. It is almost universally assumed that the dual graph is planar.
\item \textbf{Boundary shape of the dual graph.} The dual graph could be in the shape of a large square, rectangle \cite{JTF_tilings, Akitaya_et_al}, triangle \cite{cannon_irreducibility, Akitaya_et_al}, or other shape \cite{Thurston01101990,KENYON1996191}. Periodic boundary conditions have also been considered, corresponding to a map on a torus \cite{moessner2001resonating,JTF_tilings}.
\item \textbf{Node and district populations.} In the simplest models, all nodes have population 1 and each district must have a fixed size. In real-world applications, node populations vary (since each precinct or census block can have a different population), and a districting map is legal if all the district populations are equal up to a fixed error tolerance \cite{RecomMGGG}. Other variations include allowing each node to have either population 1 or population 2 \cite{JTF_tilings}, or allowing an error tolerance of $\pm 1$ in district population \cite{cannon_irreducibility, Akitaya_et_al}.
\item \textbf{Number and size of districts.} At the smallest extreme, each district could be composed of exactly 2 nodes in the dual graph. At the largest extreme, the entire dual graph could be divided into 3 large districts \cite{Akitaya_et_al, cannon_irreducibility}. (If there are only 2 large districts, then the ReCom metagraph is trivially connected.) More generally, one can consider many small districts of fixed population \cite{JTF_tilings}, or a fixed number of districts of large population, or an intermediate regime such as a dual graph with order $n^2$ nodes that is divided into approximately $n$ districts of size $n$ each \cite{Donnay21102025, JTF_tilings}.
\end{itemize}

Different choices in these categories can lead to either positive or negative answers to the key question of metagraph connectivity. For example, Tucker-Foltz \cite{JTF_tilings} considers the case where the dual graph is a square lattice in the shape of a large $n \times n$ square, all nodes have population 1, all districts have exactly equal size, and the number of districts is a fixed constant. He conjectures that under these assumptions, the metagraph is connected for sufficiently large $n$. He then provides an explicit construction to show that this conjecture is false if the nodes are allowed to have population 1 or 2. Likewise, Cannon \cite{cannon_irreducibility} proves that the metagraph is connected when the dual graph is a triangular lattice in the shape of a large triangle, all nodes have population 1, district size is allowed to vary by $\pm 1$, the districts are simply connected, and there are 3 districts. This is an impressive positive result whose robustness under changes in the assumptions is unclear.  A stronger version of this result (removing restrictions on district sizes and size of the triangular map, and removing the requirement that districts be simply connected) was proved more succinctly in \cite{Akitaya_et_al}.  The authors of \cite{Akitaya_et_al} also proved a parallel result in the case where the underlying dual graph is a rectangular subset of the grid graph.

We note that the constructions of \cite{JTF_tilings} only work when district sizes are required to stay constant, while the proof techniques of \cite{cannon_irreducibility, Akitaya_et_al} only work when district sizes may vary by $\pm 1$. All the results in this paper are proved under the assumption that the district sizes must stay constant.

\subsection{Overview of main results}\label{sec:intro_main_results}

Our work is focused in two directions: positive results that provide conditions under which the metagraph is connected, and negative constructions that explore the fragility of metagraph connectivity under small changes in the assumptions. We assume throughout that all dual graph nodes have population 1 and that all districts must be of exactly equal size with no error tolerance.

We primarily consider the case where the underlying dual graph is a triangulation, meaning a planar graph with all inner faces being triangles.  We focus on triangulation graphs in part because in our experience (and the anecdotal experience of many other researchers) many of the dual graphs under consideration when studying redistricting maps happen to be nearly triangular\footnote{By ``nearly triangular'' we mean that the dual graph is typically planar outside of a few regions that are complicated by anomalies such as water adjacencies or disconnected precincts, and most of the interior faces are triangles.}.  In addition, the question of metagraph connectivity can be somewhat more subtle on a triangular graph than on the square grid, as we demonstrate below.

Our first main result concerns the case where the districts have size 2, that is, each district consists of 2 dual graph nodes. In this setting, a legal districting map is the same thing as a \emph{dimer tiling}, or perfect matching, of the dual graph. (A dimer tiling on the square lattice is known as a \emph{domino tiling}.) Thurston \cite{Thurston01101990} used Conway's tiling groups to show that for domino tilings on the square lattice, the metagraph is always connected so long as the region being tiled is simply connected.

We consider dimer tilings on the triangular lattice. Here, the situation is more complicated: there are dual graph boundary shapes for which the metagraph is not connected.  The example of the metagraph for the triangle of side length 3 is the most basic example (as Cannon also noted \cite{cannon_irreducibility}); it is two isolated nodes, and can be seen in \cref{fig:triangle-3-metagraph}.

We also share a more interesting family of examples, in which the metagraph has at least two connected components which are not graph isomorphisms of each other.  \Cref{fig:herringbone} shows two dimer tilings of the same dual graph. The left tiling is a \emph{locked configuration} (whose pattern could be repeated in a larger region). In general, a locked configuration is a districting map such that, if any two adjacent districts are merged, the only way to split them into connected districts of equal size is to re-form the original two districts. In other words, a locked configuration is an isolated vertex in the metagraph. The tiling on the right of the same dual graph clearly has other tilings that can be reached by nontrivial ReCom moves.  Thus, \Cref{fig:herringbone} demonstrates a family of metagraphs\footnote{A \emph{family} because this pattern can be repeated on a larger scale.} whose dual graph is disconnected, with not all connected components isomorphic to each other.\footnote{Given Thurston's connectivity result, one might wonder what happens when trying to translate \Cref{fig:herringbone} into the square lattice. It turns out that the locked configuration on the left can be drawn on the square lattice, but it is the only legal dimer tiling of its underlying graph, so the metagraph only has one vertex.}

\begin{figure}[h]
    \centering
    \begin{tabular}{cc}
    \includegraphics[width=3.3cm]{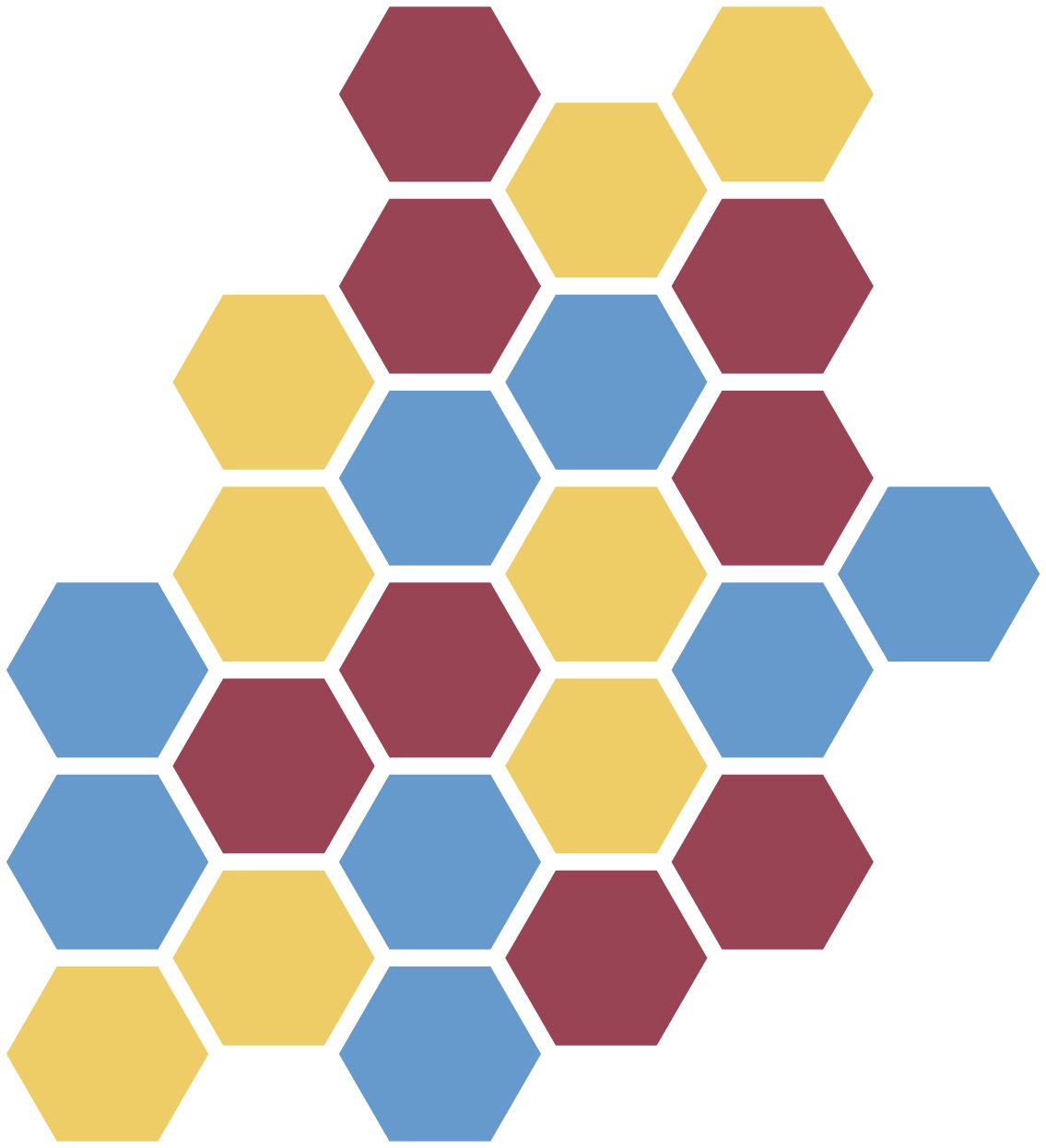}
    &
    \includegraphics[width=3.3cm]{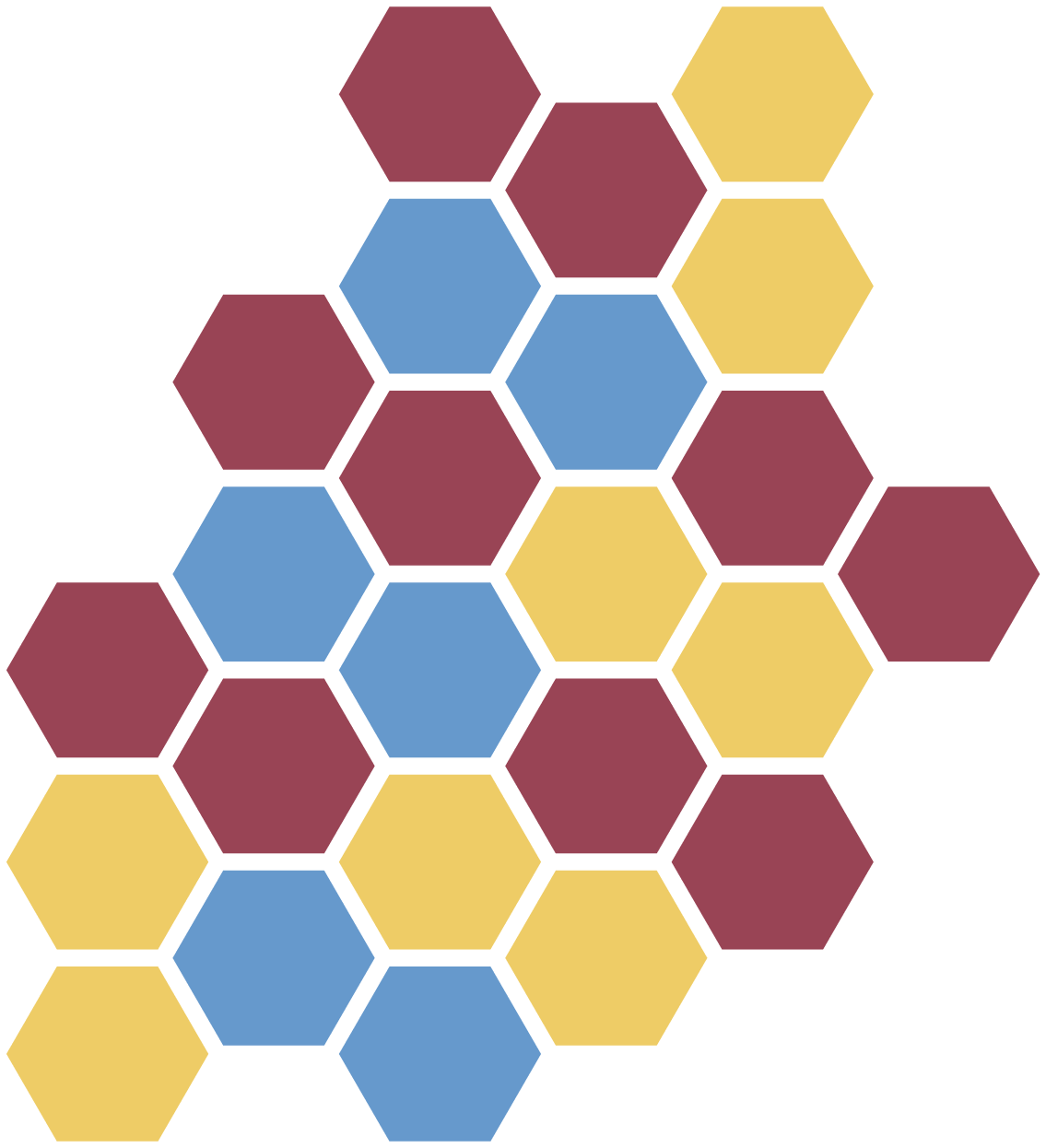}
    \\
    \includegraphics[width=2.6cm]{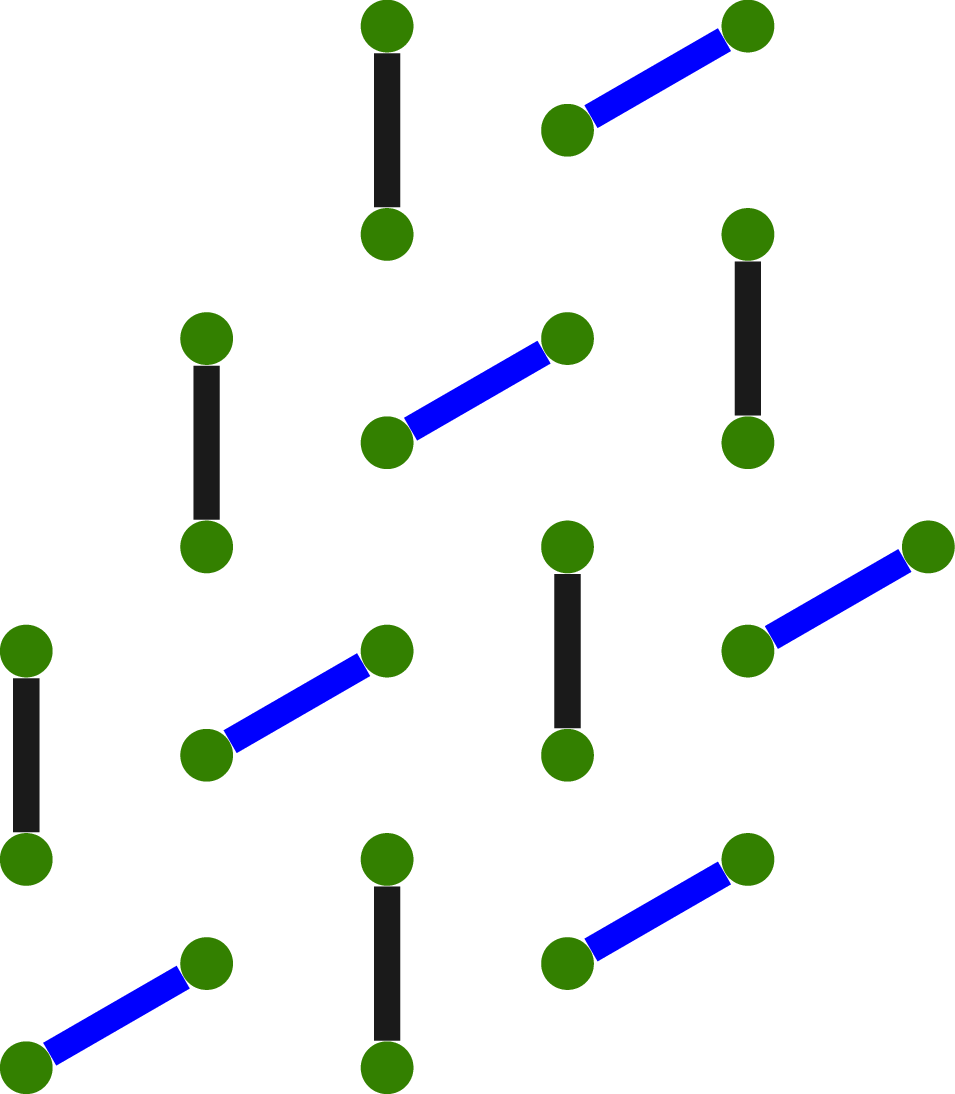}
    &
    \includegraphics[width=2.6cm]{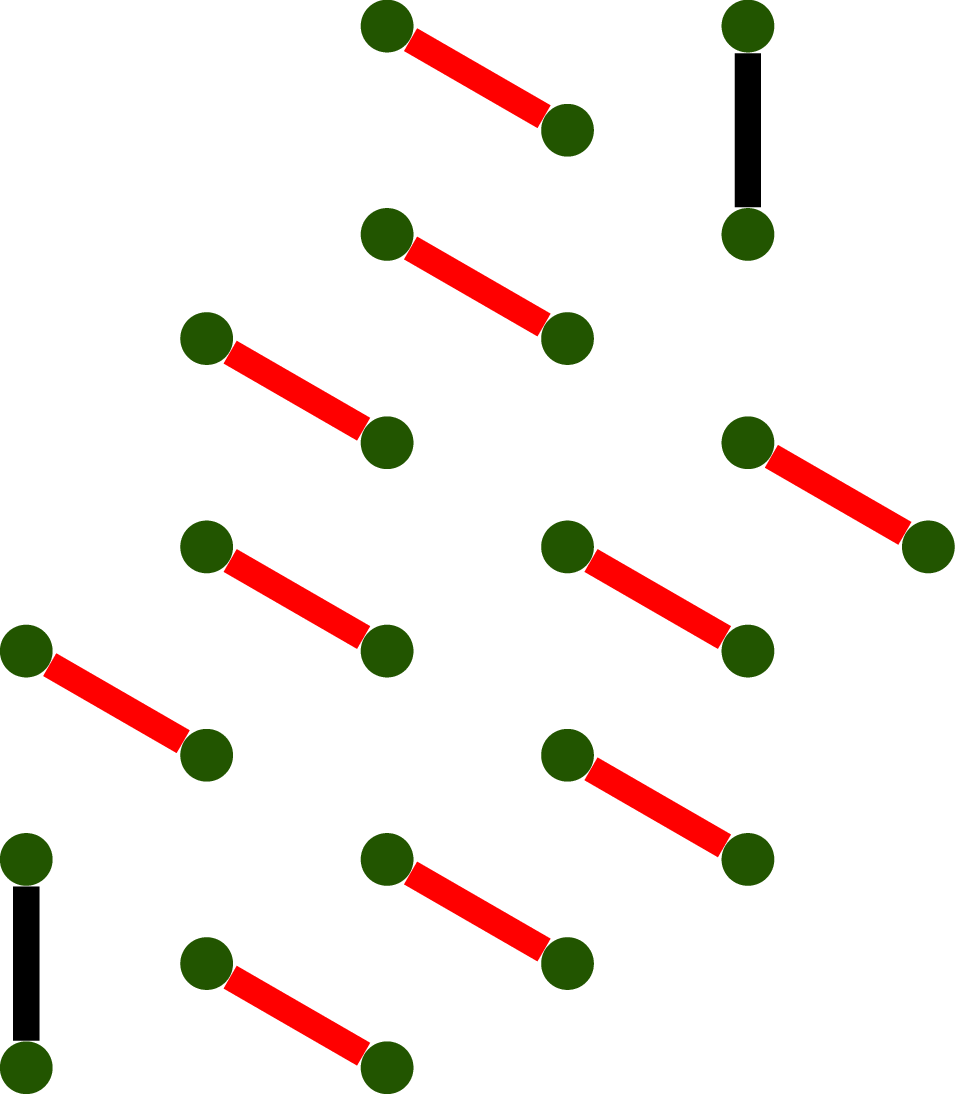}
    \end{tabular}
    
    \caption{Top left and top right: Two tilings of a region.  The tiling on the left is locked whereas the one on the right is not.  As the tiling on the left is locked and there are at least two tilings, the metagraph of this region is disconnected. Bottom left and bottom right: Schematic diagrams for the same two tilings, where each hexagonal cell is represented by a node and cells in the same tile are connected by a line segment. We will pass back and forth between these two equivalent visual representations throughout the paper.}
    \label{fig:herringbone}
\end{figure}

In contrast with this example, \Cref{thm:main_result_triangular_lattice} proves that the metagraph for dimer tilings on the triangular lattice is connected when the dual graph is in the shape of a large triangle (of side length at least 5). See \Cref{fig:example-triangle-tiling} for an example. \Cref{thm:main_result_triangular_lattice} is of both theoretical and practical interest. From a theoretical point of view, it adds to the short list of positive results guaranteeing metagraph connectivity under specific assumptions, and it provides the first natural family of examples in which the metagraph is either connected or disconnected depending on the boundary shape of the dual graph while all other features remain the same.

\begin{figure}[h]
    \centering
    \includegraphics[width=3.6cm]{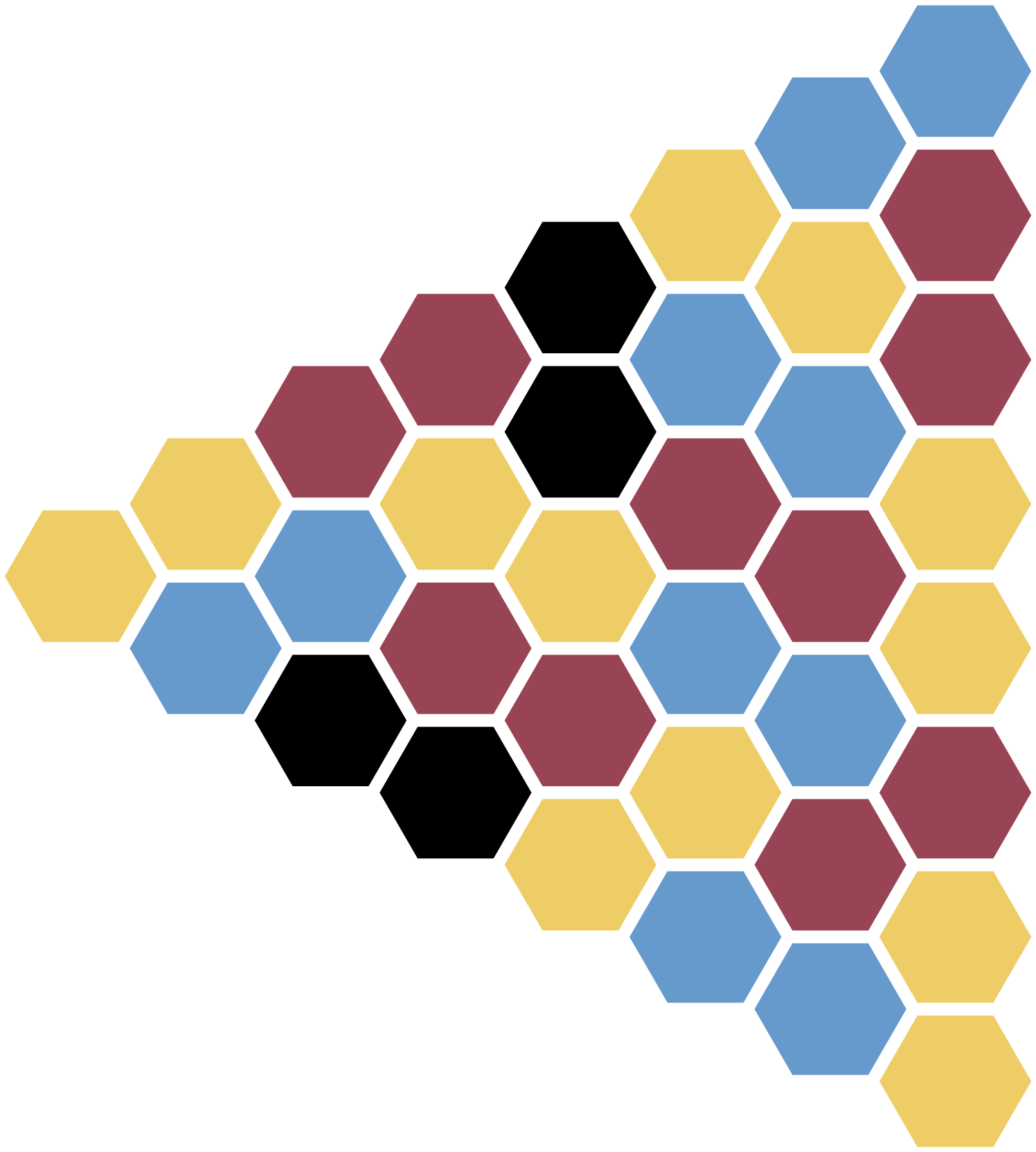}
    \caption{An example tiling of the triangle of side-length 8.}
    \label{fig:example-triangle-tiling}
\end{figure}

Practically, the case of size-2 districts has very real application in redistricting. Indeed, 15 states have a \emph{nesting} requirement (or preference) for their state senate districts:  they require (or prefer) that each state senate district be constructed from 2 (or for some states 3) state house districts merged together \cite{HouseSenateMergingStates}.  Thus, \Cref{thm:main_result_triangular_lattice} effectively says that if a state had this 2:1 nesting requirement and its underlying dual graph for the state house districts were a triangular subset of the triangular lattice, then every possible state senate districting map could be produced by starting with a matching of state house districts and performing a sequence of ReCom moves.

Our second main result is a negative counterpart to Cannon's positive connectivity result on the triangular lattice \cite{cannon_irreducibility} (re-proved in \cite{Akitaya_et_al}). We make the following changes to her assumptions:
\begin{itemize}
\item Cannon requires the dual graph to be a triangular lattice. We relax this condition but still require the dual graph to be a triangulation.
\item Cannon considers the case of 3 districts. We consider 6 districts.
\item Cannon allows variation of $\pm 1$ in the district size. We require all districts to have the same fixed size.
\item The boundaries of Cannon's shapes are equilateral triangles.  We allow for a family of shapes, which includes hexagons whose adjacent side lengths are, in order around the boundary: $2k+1$, $2k-1$, $2k+1$, $2k-1$, $2k+1$, $2k-1$.
\end{itemize}
In this setting, our \Cref{thm:locked-6} constructs an infinite family of dual graphs, each of which admits a locked configuration. \Cref{fig:locked-6} shows the construction.

\afterpage{%
\begin{figure}[!htbp]
\centering
\includegraphics[width=0.6\textwidth]{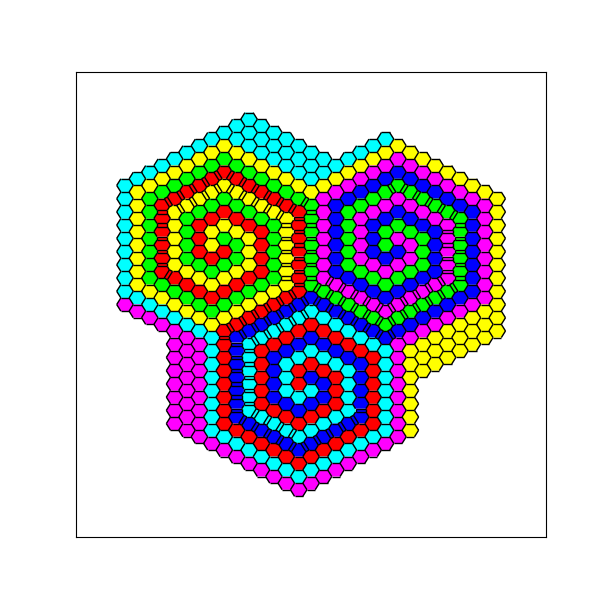} \\
\includegraphics[width=0.6\textwidth]{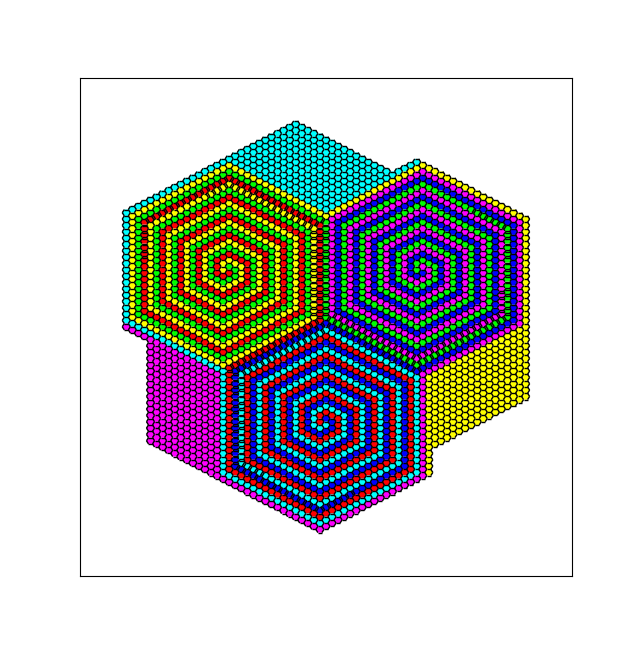}
\caption{Construction of \Cref{thm:locked-6}: Locked configurations with six districts where the dual graph is a bounded degree triangulation. Top: Case $k=8$. Bottom: Case $k=17$.}
\label{fig:locked-6}
\end{figure}
\clearpage
}

Evidently, if there is a locked configuration then the metagraph must be disconnected (unless the underlying dual graph has only one legal districting map, which is not the case here). Taken together, \Cref{thm:locked-6} and the positive result of \cite{cannon_irreducibility} provide another demonstration of how sensitive the property of metagraph connectivity is to the specific assumptions that are imposed.

It is instructive to compare \Cref{thm:locked-6} with a similar locked configuration constructed by Tucker-Foltz \cite{JTF_tilings} on the square lattice with some nodes having population 1 and others having population 2. Both locked configuations consist of 6 equally sized districts. Besides the differences in dual graphs (square lattice versus triangulation) and node populations (a mixture of 1's and 2's versus all 1's\footnote{It is worth noting that our construction could be done on the triangular lattice (without subdividing it) by allowing some nodes to have population 1, and some to have population 2. See Step 3 in \Cref{sec:locked_few_districts} for further discussion.}), there is another distinction of interest. Each individual district in the construction of \cite{JTF_tilings} is ``snake-shaped'': it forms a path in the dual graph. In \Cref{thm:locked-6}, by contrast, three of the six districts are snake-shaped (red, green, and dark blue in \Cref{fig:locked-6}) and the other three are not (light blue, purple, and yellow), as they have a region comprising about half the nodes of the district which forms a parallelogram. Putting together the three solid regions of light blue, purple, and yellow, those regions cover $1/4$  of the dual graph asymptotically as the number of nodes goes to infinity.

This feature of \Cref{thm:locked-6} is noteworthy because MCMC ReCom algorithms used in practice are designed to favor compact districts and stay away from snakelike shapes \cite{RecomMGGG}. We are far from drawing any direct connection between the theoretical existence of locked configurations in idealized models and the mixing properties of ReCom Markov chains applied to real-world maps. Still, \Cref{thm:locked-6} demonstrates that it is possible to create a locked configuration where $1/4$  of the map looks like portions of the typical compact shapes of an actual district map. It is an open question whether this $1/4$   fraction could be increased by a different construction.

Could one construct a family of locked configurations such as in \Cref{thm:locked-6} or in \cite{JTF_tilings} with fewer than 6 districts? To discuss this question, we define the \emph{district graph} of a configuration as follows. The vertices of the district graph correspond to the districts in the configuration, and two vertices of the district graph are connected by an edge if and only if the corresponding districts in the configuration share a border. \Cref{fig:district-graph} shows the district graph of the configuration in the top panel of \Cref{fig:locked-6}.

\begin{figure}[h]
    \centering
    \begin{tikzpicture}[scale=0.6]
        \node[circle, draw, fill=aqua] (1) at (0,4) {};
        \node[circle, draw, fill=pink] (2) at (-2.31,0) {};
        \node[circle, draw, fill=yellow] (3) at (2.31,0) {};
        \node[circle, draw, fill=blue] (4) at (-0.58,1.67) {};
        \node[circle, draw, fill=red] (5) at (0.58,1.67) {};
        \node[circle, draw, fill=green] (6) at (0,0.67) {};

        \path[draw] (1) to[bend right] (2);
        \path[draw] (1) to[bend left] (3);
        \path[draw] (2) to[bend right] (3);
        \path[draw] (1) to (4);
        \path[draw] (1) to (5);
        \path[draw] (2) to (4);
        \path[draw] (2) to (6);
        \path[draw] (3) to (5);
        \path[draw] (3) to (6);
        \path[draw] (4) to (5);
        \path[draw] (4) to (6);
        \path[draw] (5) to (6);
    \end{tikzpicture}
    \caption{District graph of the configuration in the top panel of \Cref{fig:locked-6}.}
    \label{fig:district-graph}
\end{figure}

The locked configurations constructed in \Cref{thm:locked-6} and in Section 6 of \cite{JTF_tilings} have exactly the same district graph as \Cref{fig:district-graph}. This is because the constructions require each district to border at least 4 others to make the configuration locked, and the district graph must be planar. \Cref{fig:district-graph} is the smallest planar graph in which every vertex has degree at least 4. (The only graph on 5 vertices with minimum degree 4 is the complete graph $K_5$, which is nonplanar.)

To construct a locked configuration with fewer than 6 districts, at least one of them would need to have degree 3 or less in the district graph. This is possible in very small cases (see \Cref{fig:triangle-3-metagraph}), but it is not known whether there is an infinite family of such examples.

A common feature of the locked configurations whose district graph is shown in \Cref{fig:district-graph} is that if one allows the district size to vary by $\pm 1$, the configurations are no longer locked. If one wanted to construct a configuration that is still locked when allowing for minor fluctuations in the district size, a good starting point might be a district graph with minimum degree 5.

Our last main result, \Cref{thm:3_by_n_rectangle}, considers a square lattice dual graph. We assume as always that each node has population 1 and that all districts are the same size. Tucker-Foltz \cite{JTF_tilings} conjectured that if the dual graph is an $n \times n$ square grid split into $k$ equal districts, where $k$ is fixed, then the metagraph is connected for sufficiently large $n$. We prove in \Cref{thm:3_by_n_rectangle} that the metagraph is indeed connected when the dual graph is a $3 \times n$ square grid split into 3 equal districts.

\subsection{Prior work}\label{sec:prior_metagraph}

There is a substantial literature on dimer tilings (and perfect matchings). Kasteleyn \cite{kasteleyn1967graph} provided an exact formula for the number of dimer tilings on a given graph. The \emph{dimer model} in statistical physics \cite{kenyon2003introduction} is the study of random dimer tilings on a graph and has connections with conformal probability, random matrix theory, and other branches of mathematics.

As previously mentioned, Thurston \cite{Thurston01101990} considered domino tilings on subgraphs of the square lattice and showed that the metagraph is always connected so long as the region being tiled is simply connected. Kenyon and R\'{e}mila \cite{KENYON1996191} gave a linear-time algorithm that determines if a subgraph of the triangular lattice has a perfect matching (and produces a matching, if one exists).  In \cite{KENYON1996191}, they also proved that if $G$ is a subgraph of the triangular lattice with a perfect matching, then any two perfect matchings of $G$ can be transformed into each other by a linear number of ReCom moves, where combinations of 2, 3, or 4 districts together are allowed.  Hartarsky, Lichev, and Toninelli \cite[Theorem 8]{HLT24} refined this result by showing that only ReCom moves combining 2 or 3 districts are needed when the underlying graph is in the shape of a large parallelogram.  We show in \Cref{thm:main_result_triangular_lattice} that when the underlying graph is in the shape of a large triangle, in fact only combinations of 2 districts are required to connect the metagraph.\footnote{Whether the graph has a parallelogram or triangle shape is simply a matter of convenience. It is pointed out in \cite{HLT24} that their method also works for triangles, and our approach to \Cref{thm:main_result_triangular_lattice} very likely works for parallelograms.}  The proof method of \cite{HLT24} is a double induction, quite similar to ours.  A key difference is that the authors of \cite{HLT24} use the 3-district ReCom moves to set up a relatively short case analysis that is done by hand.  Using only 2-district ReCom moves, one encounters examples such as \Cref{fig:failing-cover} which show that some serious computational power is necessary.

Moessner and Sondhi \cite{moessner2001resonating} considered dimer tilings on the triangular lattice with periodic boundary conditions. In that setting, they found two parity invariants that split the metagraph into four sectors (odd-odd, odd-even, even-odd, even-even) which are closed under ReCom moves. Additionally, they constructed 12 locked configurations of the same form as the left panel of \Cref{fig:herringbone}. They conjectured that the four metagraph sectors are connected if one excludes the 12 isolated vertices corresponding to the locked configurations. To our knowledge, this conjecture remains open. It could be resolved by an approach similar to our proof of \Cref{thm:main_result_triangular_lattice}, but significantly more computational power would be required.

Cannon \cite{cannon_irreducibility} considered metagraphs in the triangular lattice with three large districts.  She meticulously showed that if a triangular subset of the triangular lattice is partitioned into three simply connected components of size $k_1, k_2$, and $k_3$, and any ReCom move is allowed that keeps those sizes within 1 of ideal (the sizes are $k_i'$ where $k_i' \in \{k_i-1, k_i, k_i+1\}$ for $i = 1, 2, 3$), then the corresponding metagraph is connected. Our \Cref{thm:locked-6} is a negative counterpoint to Cannon's positive result, as we discussed in \Cref{sec:intro_main_results}.

Akitaya et al \cite{Akitaya_et_al} proved a stronger version of Cannon's result \cite{cannon_irreducibility}, removing the restriction on district sizes, the restriction on size of the subset of the triangular lattice, and the requirement that districts be simply connected. Akitaya et al also prove the same result when the underlying graph is a rectangular subset of the grid graph.  We note that the allowance that district sizes be allowed to diverge by $\pm 1$ from the ideal size is key to both Cannon and Akitaya et al's proofs.  It also distinguishes Akitaya et al's grid graph result from our \Cref{thm:3_by_n_rectangle}.

Tucker-Foltz \cite{JTF_tilings} conducted an in-depth study of many different metagraphs with a focus on locked configurations. Here are some of his results:
\begin{itemize}
\item There are arbitrarily large locked configurations on rectangular subgraphs of the square lattice for districts of size 3 and districts of size 4.
\item There are infinite families of locked configurations with approximately $n$ districts of size $n$ each on both the square lattice and the triangular lattice with periodic boundary conditions.
\item There is an infinite family of locked configurations with 6 districts on $n \times n$ grid graphs where each node has population 1 or 2. \Cref{sec:intro_main_results} discussed the comparison between this construction and our \Cref{thm:locked-6}.
\item When the dual graph is the $6 \times 6$ grid and each district has size 3, the metagraph consists of many different connected components (137 in total) of all sizes from the set $\{1, 2, 8, 16, 19, 20, 68, 199, 235, 384, 73738\}$.  Some of the connected components of the same size are isomorphic, and some are not.  As we shall see in \Cref{sec:invariants}, the study of ReCom invariants can distinguish between some of these connected components, but certainly cannot distinguish between all of them.  This small example is a fascinating glimpse into the complexity of these ReCom metagraphs.
\end{itemize}

We note that a key component in \cite{cannon_irreducibility} is an ``Unwinding Lemma'' that shows how to use ReCom moves to disentangle districts which may have long, intertwining arms.  This idea of having long entangled districts which cannot be re-drawn with ReCom moves was a key feature of several of the locked configurations in \cite{JTF_tilings}.  We also explore the impact of intertwined districts in \Cref{sec:locked_few_districts,sec:thin_rectangle}.

For other related work, we point the reader to the references in Section 1.1 of \cite{JTF_tilings}.

\subsection{Outline}

The rest of this paper is organized as follows:  We discuss and prove \Cref{thm:main_result_triangular_lattice} in \Cref{sec:size_two_districts}.  We discuss and prove \Cref{thm:locked-6} in \Cref{sec:locked_few_districts}, and we prove \Cref{thm:3_by_n_rectangle} in \Cref{sec:thin_rectangle}.  We discuss conclusions and open questions in \Cref{sec:conclusions}.

We also include, in \Cref{sec:invariants}, a discussion of algebraic invariants on connected components of ReCom metagraphs for the interested reader.  These invariants were studied in \cite{ExploringMetagraphs} in the case of the $6 \times 6$ grid graph, and we found that exploration to be interesting and motivating.  Perhaps not surprisingly, just as in \cite{ExploringMetagraphs}, we see that these invariants cannot distinguish between all connected components in a metagraph for a triangular subset of the triangular lattice.  Nevertheless, these invariants can pick up nuances in the connected components of such metagraphs, and we believe this study is worth further exploration.  

Lastly, \Cref{sec:pseudocode} includes the pseudocode and a description of the algorithms we used for computer search as part of our proof of \Cref{thm:main_result_triangular_lattice}.

\section{Triangular lattice with districts of size 2}\label{sec:size_two_districts}

Recall that the \emph{ReCom metagraph} is the graph whose nodes are tilings (districting maps) and edges are between nodes that can be obtained from each other by a single ReCom move.  Here we investigate ReCom metagraphs of dimer tilings of triangle-shaped subsets of the triangular lattice.  We show that the metagraph is connected for all triangles with
an even number of cells except for triangles of side-length 3 and 4.

We adopt the convention that our triangles are oriented with one vertex on the left and a vertical column on the right. In addition, we must sometimes remove the bottom corner of the triangle to ensure that the total number of cells is even. See \Cref{fig:clipped-full-triangles} for example tilings and triangles.

Our proof of metagraph connectivity is an inductive argument based on the side length. The inductive step comes down to a large but finite case analysis which is verified by computer search.

\subsection{Clipped and full triangles}

A necessary condition for a triangle of side-length $n$ to be tiled with tiles
of size $k$ is that $k$ must divide $\tfrac{n(n + 1)}{2}$, the number of nodes
in the triangle.  As we consider $k = 2$, that means that $\tfrac{n(n + 1)}{2}$
must be even, which is true when $n \equiv 0, 3 \pmod 4$.

To extend our argument to triangles of any size, we remove the bottom corner of
the triangle when $n \equiv 1, 2 \pmod 4$. We call triangles without a bottom corner
\emph{clipped triangles}. We call triangles that are not clipped \emph{full
  triangles} whenever we need to distinguish them.  We use the following
definition to make claims about triangles that are either clipped or full, depending on
the side length:

\begin{Definition}[\Ctriangle]
  An \ctriangle of side length $n$ is the full triangle with side length $n$ if $n \equiv 0, 3 \pmod 4$, and it is the clipped triangle with side length $n$ if $n \equiv 1, 2 \pmod 4$.
\end{Definition}

We emphasize that the word ``even'' in ``\ctriangle'' means that the number of nodes is even, not that the side length $n$ is even. \Cref{fig:clipped-full-triangles} shows example tilings for clipped and full triangles.

\begin{figure}[h]
  \centering

  \begin{subfigure}[b]{0.20\textwidth}
    \includegraphics[width=2.25cm]{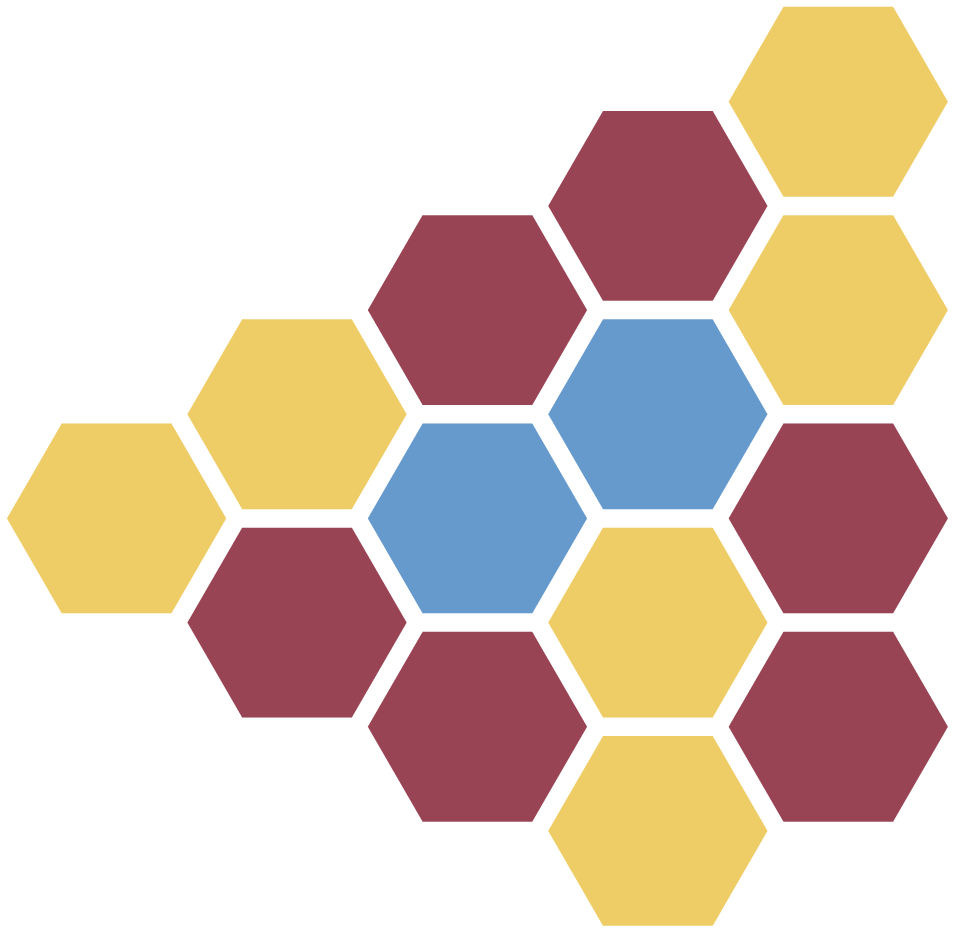}
    \caption{$5 \equiv 1 \pmod 4$\\ \centering clipped}
  \end{subfigure}
  \hfill
  \begin{subfigure}[b]{0.21\textwidth}
    \includegraphics[width=2.7cm]{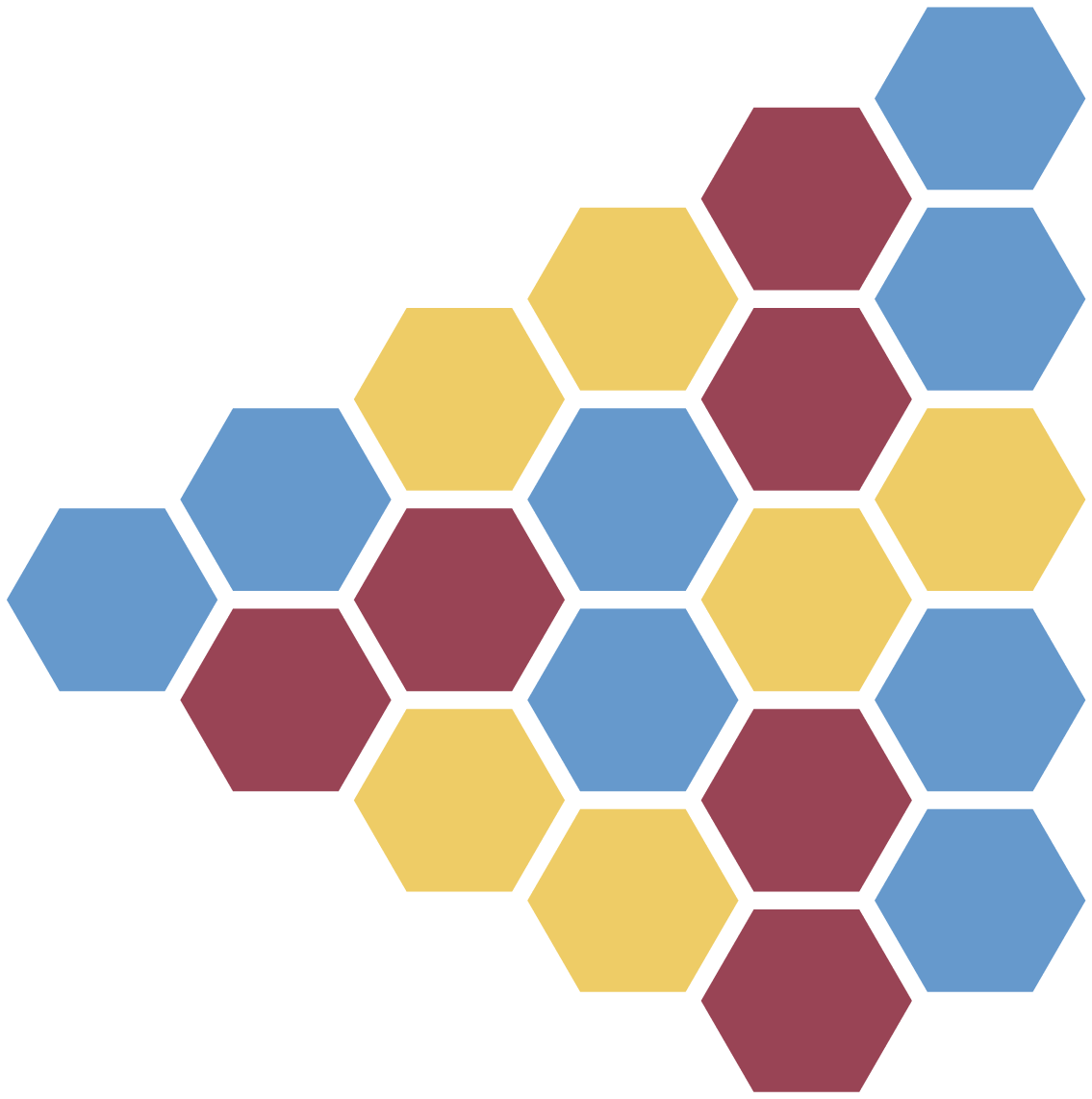}
    \caption{$6 \equiv 2 \pmod 4$\\ \centering clipped}
  \end{subfigure}
  \hfill
  \begin{subfigure}[b]{0.24\textwidth}
    \includegraphics[width=3.15cm]{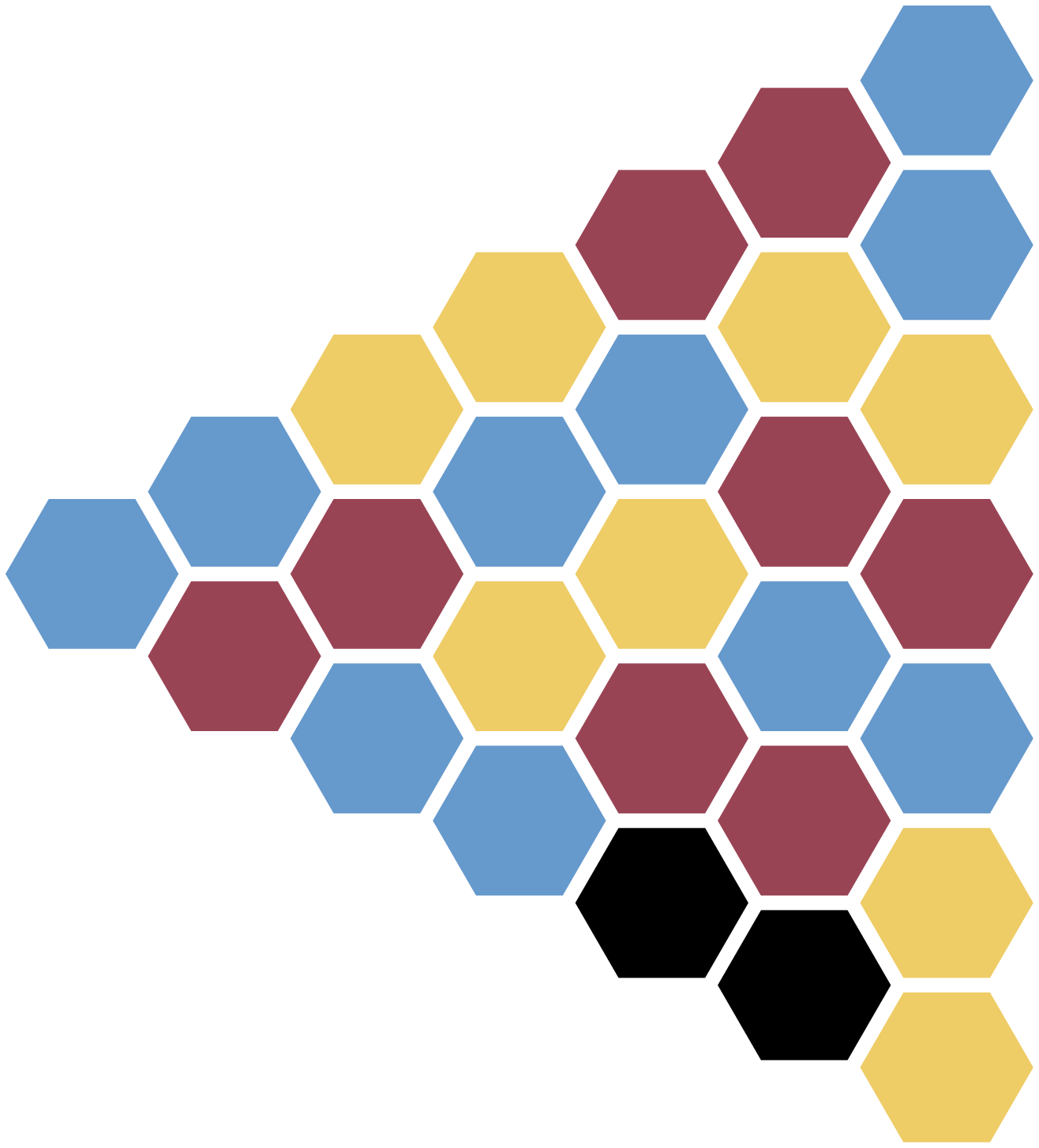}
    \caption{$7 \equiv 3 \pmod 4$\\full}
  \end{subfigure}
  \hfill
  \begin{subfigure}[b]{0.28\textwidth}
    \includegraphics[width=3.6cm]{8-pdf}
    \caption{$8 \equiv 0 \pmod 4$\\full}
  \end{subfigure}
  
  \caption{Example tilings for clipped and full triangles with different side lengths $\mod 4$.}
  \label{fig:clipped-full-triangles}
\end{figure}

\subsection{The structure of our proof}

Our main claim is the theorem below.

\begin{Theorem}\label{thm:main_result_triangular_lattice}
  For all $n$ except 3 and 4, the ReCom metagraph for dimer tilings of the \ctriangle of side length $n$ is connected.
  Moreover, the diameter of the metagraph is in $\Theta(n^2)$.
\end{Theorem}

We build up to this claim in the following manner:

\begin{itemize}
\item We use computer search to check the base cases up to side length 9.
\item We build an inductive proof based on the side length for larger side lengths.  This part of the proof also utilizes computer search.
\end{itemize}

The inductive proof rearranges the two rightmost columns of the triangle to
consist of only vertical tiles (except possibly at the very bottom). This reduces the side-length $n$ case to the
side-length $n - 2$ case.  To create vertical tiles in the two rightmost columns, we start at the top and work our way downward. The key ingredient is \Cref{lemma:inductive-gadget} below, which is illustrated by \Cref{fig:inductive-gadget}.

\begin{figure}[h]
  \centering
  \qquad\qquad\qquad\qquad \includegraphics[width=8cm]{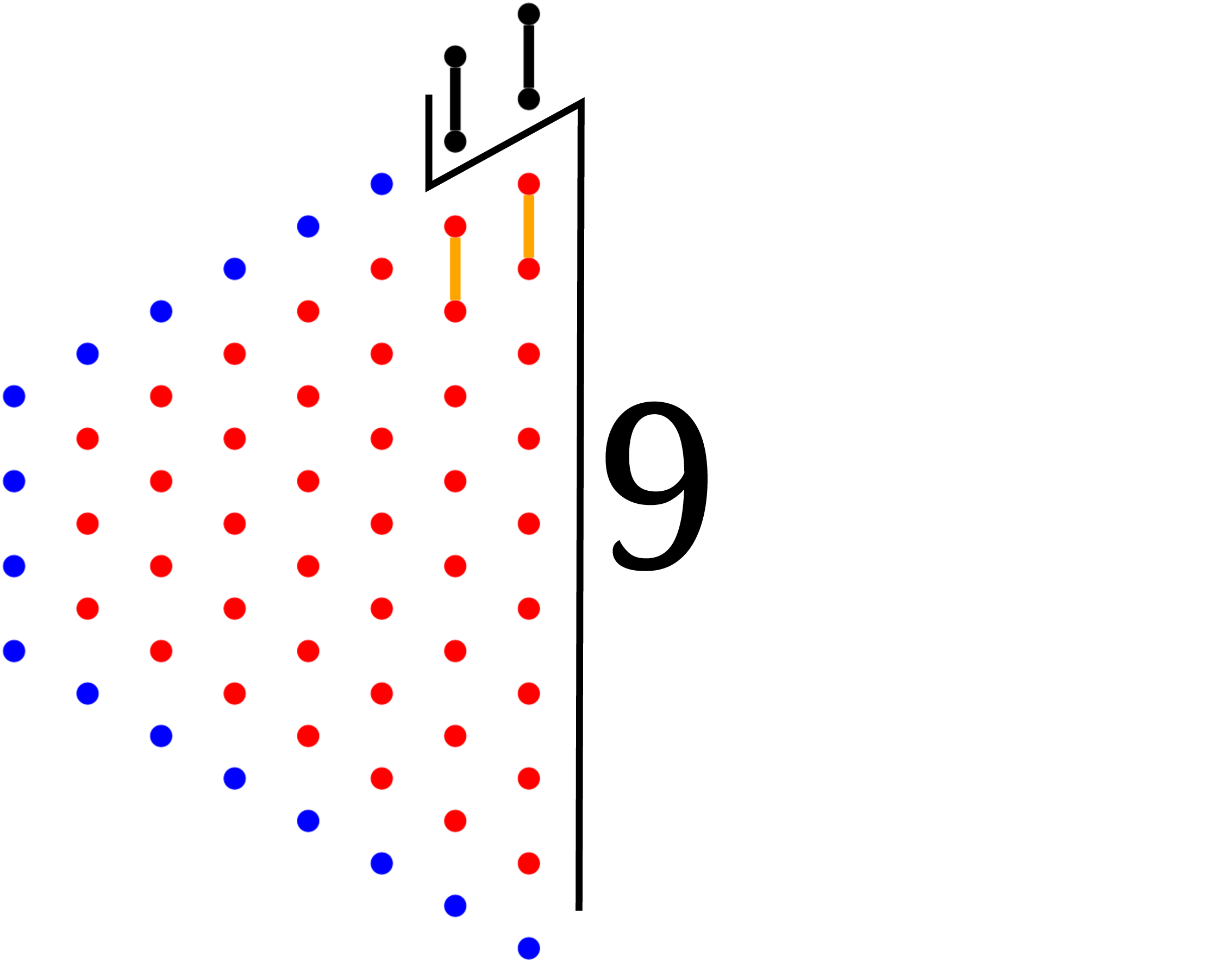}
  \vspace{8pt}
  \caption{The fixed region we use to build the two columns inductively. A ``permitted tiling'' of this region is a dimer tiling that covers all of the red nodes, any subset of the blue nodes, and none of the black nodes. The black lines indicate a hard boundary.  The black tiles are the tiles
    obtained in the previous iteration, and the orange tiles are the tiles we
    want to obtain in this inductive step.}
  \label{fig:inductive-gadget}
\end{figure}

\begin{Lemma}
\label[lemma]{lemma:inductive-gadget}
Consider \Cref{fig:inductive-gadget}.  A ``permitted tiling'' of this region is a dimer tiling that covers all of the red nodes, any subset of the blue nodes, and none of the black nodes.  Every permitted tiling of \Cref{fig:inductive-gadget} can be transformed by a series of ReCom moves into a tiling that contains the two orange tiles.
\end{Lemma}

\begin{proof}
    Via checking cases exhaustively using a computer.  This exhaustive search looks at every possible dimer tiling of all red nodes that covers any number of the blue nodes, and finds a sequence of ReCom moves from that tiling to another  tiling with the two orange vertical tiles.
    
    As this lemma pertains only to the specific graph in \Cref{fig:inductive-gadget}, it is a finite (but very large) case analysis.
    See \Cref{sec:computer-search} for details about the implementation.
\end{proof}

Suppose we are given an \ctriangle of side length $n>9$. Starting from any dimer tiling, \Cref{lemma:inductive-gadget} allows us to perform ReCom moves so that there is a vertical tile at the top of each of the two rightmost columns. Then, we use \Cref{lemma:inductive-gadget} again to create two more vertical tiles immediately below the first two. By repeatedly applying \Cref{lemma:inductive-gadget}, we obtain vertical tiles going all the way down the two rightmost columns from the top until we get close enough to the bottom corner that there is no longer enough room for \Cref{fig:inductive-gadget} to fit. At that point, we use \Cref{lemma:bottom-corner} below (illustrated by \Cref{fig:bottom-gadgets}) to complete the two rightmost columns.

\begin{figure}[h]
  \centering
  \begin{subfigure}[b]{0.23\textwidth}
    \includegraphics[width=3cm]{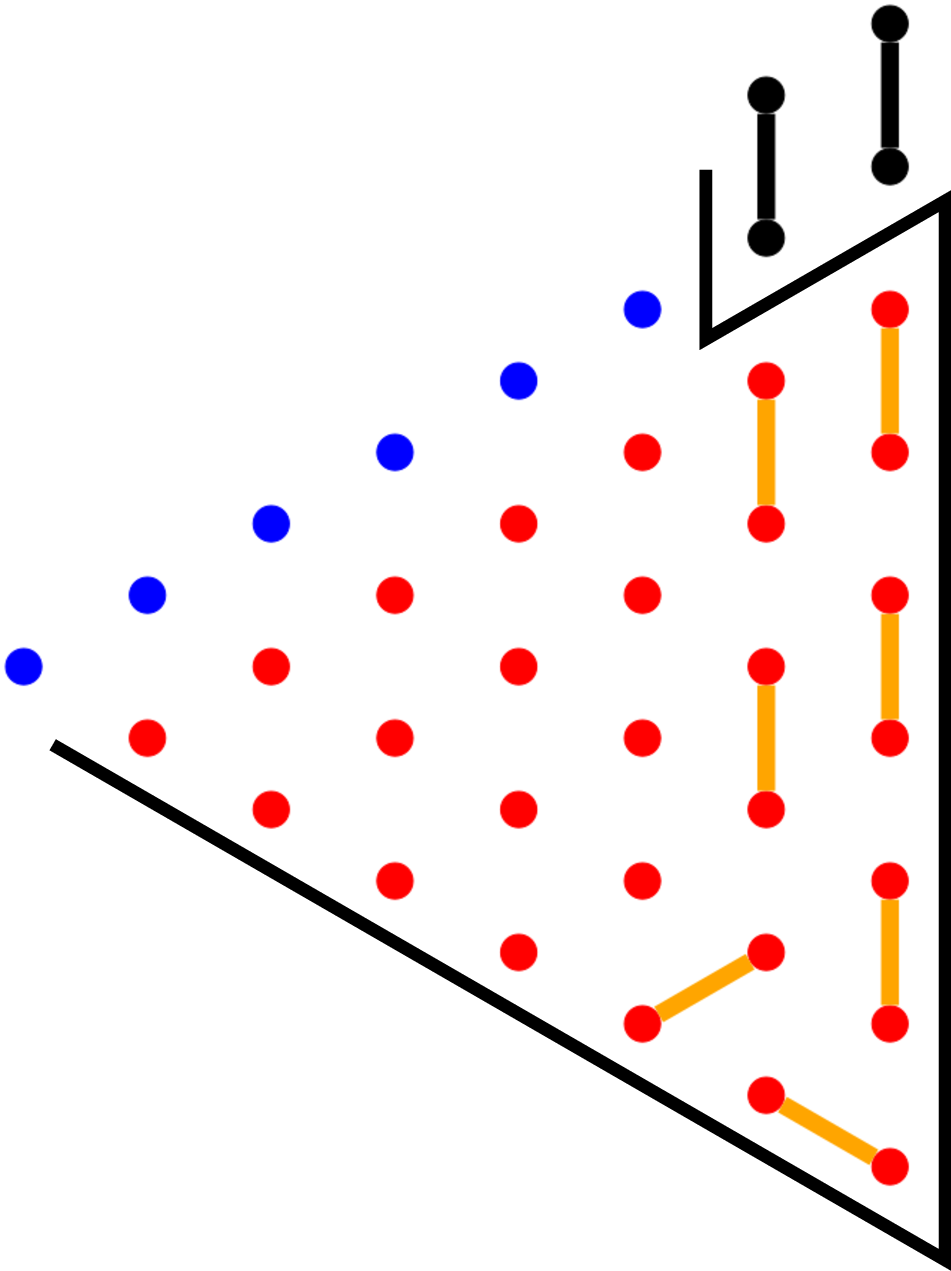}
    \caption{$n$ is odd, full}
  \end{subfigure}
  \hfill
  \begin{subfigure}[b]{0.23\textwidth}
    \includegraphics[width=3cm]{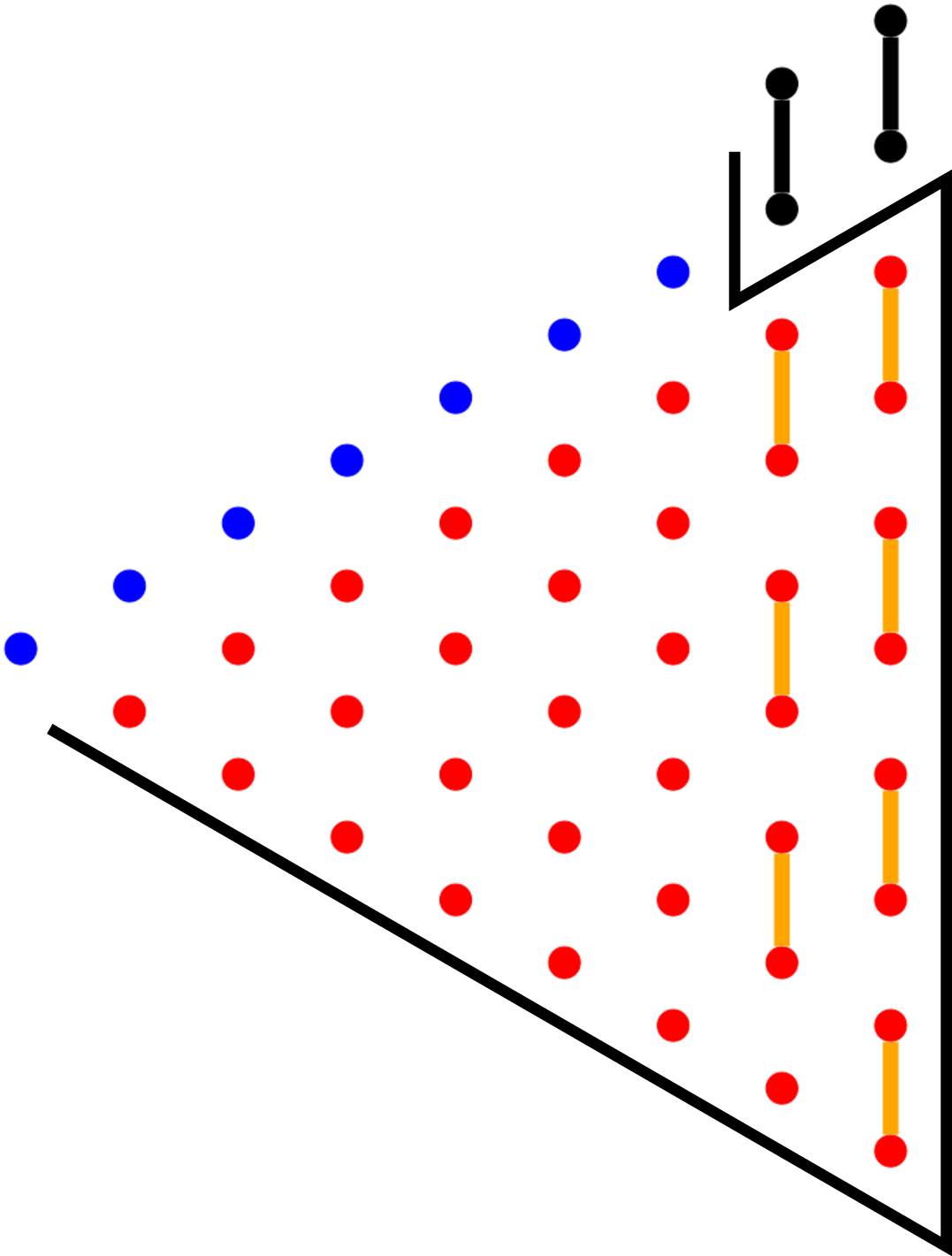}
    \caption{$n$ is even, full}
  \end{subfigure}
  \hfill
  \begin{subfigure}[b]{0.23\textwidth}
    \includegraphics[width=3cm]{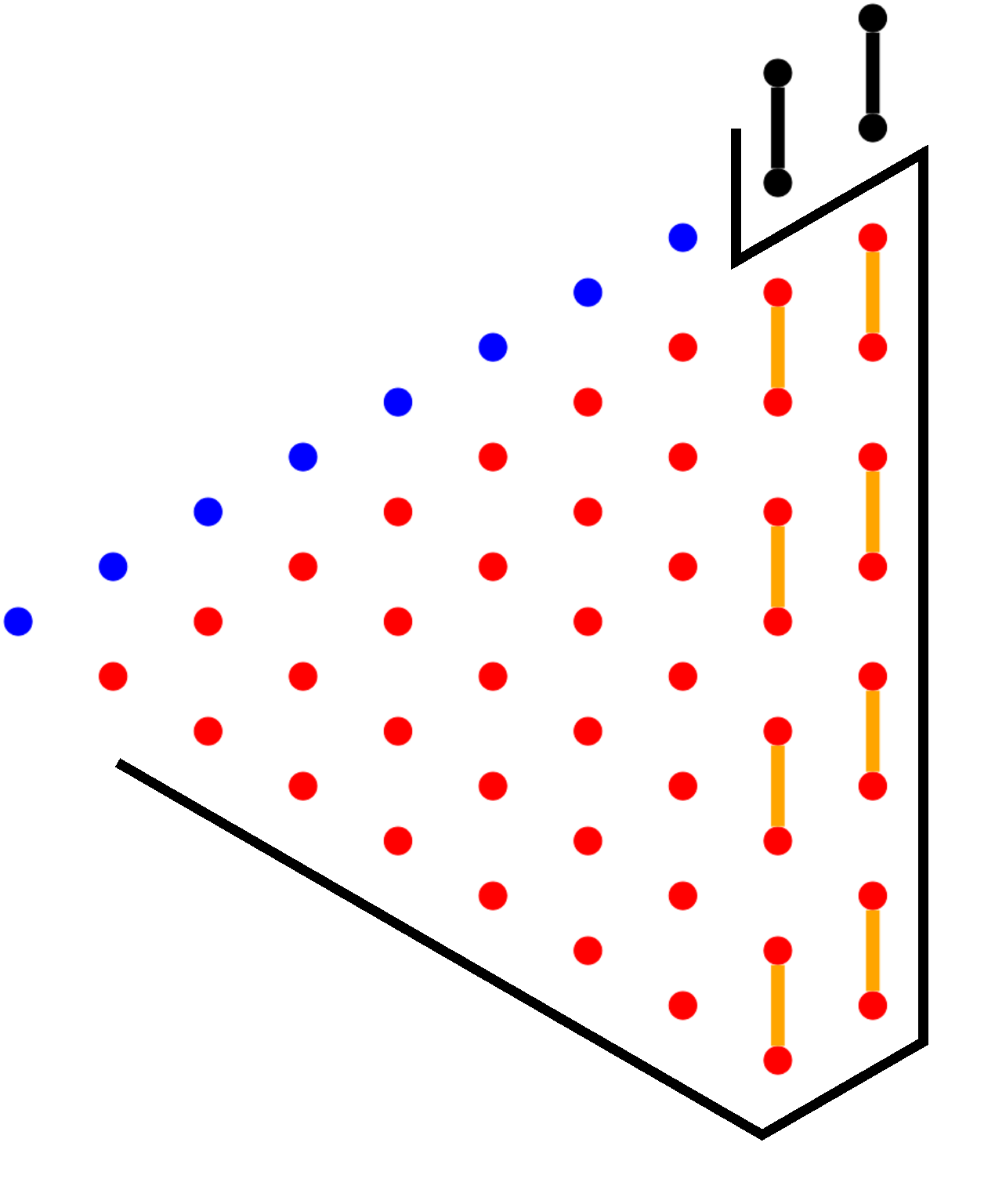}
    \caption{$n$ is odd, clipped}
  \end{subfigure}
  \hfill
  \begin{subfigure}[b]{0.23\textwidth}
    \includegraphics[width=3cm]{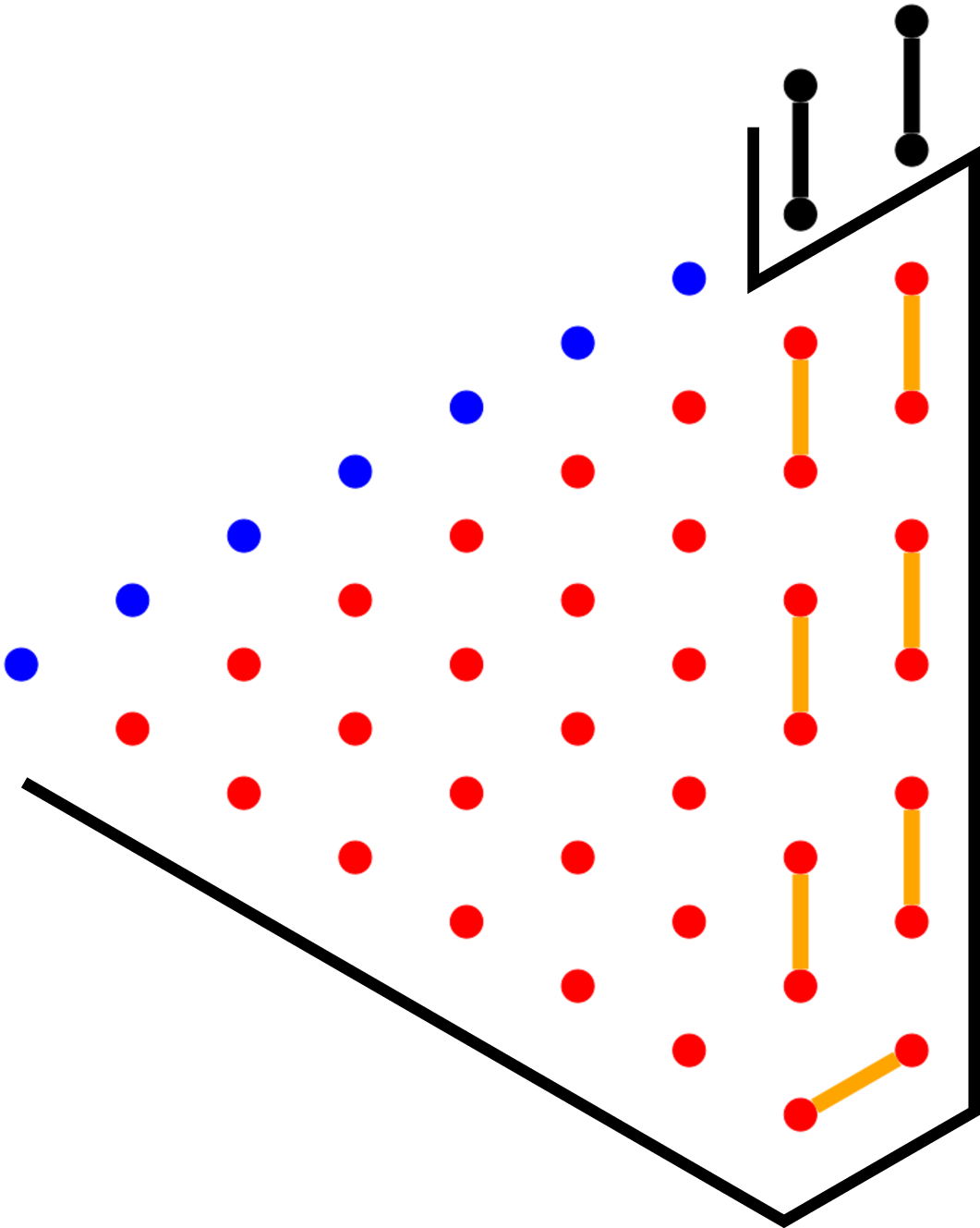}
    \caption{$n$ is even, clipped}
  \end{subfigure}
  
  \caption{The fixed regions we need to check for the bottom corner. The colors have the same meaning as \Cref{fig:inductive-gadget}. The number $n$ is the side length of the original triangle.}
  \label{fig:bottom-gadgets}
\end{figure}

\begin{Lemma}
\label[lemma]{lemma:bottom-corner}
Consider the four regions in \Cref{fig:bottom-gadgets}.  A ``permitted tiling'' of such a region is a dimer tiling that covers all of the red nodes, any subset of the blue nodes, and none of the black nodes.  For each of the four regions in \Cref{fig:bottom-gadgets}, every permitted tiling of that region can be transformed by a series of ReCom moves into a tiling that contains the orange tiles shown.
\end{Lemma}

\begin{proof}
    Via exhaustive case analysis using a computer, just as for \Cref{lemma:inductive-gadget}.
\end{proof}

Once we have completed the two rightmost columns of the triangle using \Cref{lemma:bottom-corner}, the remaining nodes form an \ctriangle of side-length $n-2$. Thus, we have reduced to the side-length $n-2$ case of \Cref{thm:main_result_triangular_lattice}, and the proof is complete by induction. See \Cref{sec:tile-2-inductive-proof} for the full details of the proof.

\subsection{Checking metagraphs of covers}
\label{sec:computer-search}

Throughout our proof, we consider several subsets of the triangle to work in,
and we want to show that all possible tilings of a chosen fixed region can be
transformed via ReCom moves to satisfy a particular tiling constraint.  To this end, we use
computer-assisted search.

To perform our case analysis, we follow a semi-na\"ive algorithm. Given any permitted tiling of \Cref{fig:inductive-gadget} (or any other fixed region), its ``cover'' is the set of nodes that are tiled. The first step in our algorithm is to generate all the possible covers. These are simply the sets of even cardinality that are the union of the red nodes with some subset of the blue nodes. Then, we generate and check the metagraphs of all the covers
in parallel. The success condition is that for every cover, each connected component of its metagraph must contain a tiling that includes the orange tiles. We also cache ReCom moves and optimize the way we represent tilings
in memory, as our process is heavily memory-bound.  The largest region we checked
(the inductive region in \Cref{fig:inductive-gadget}) took approximately 17 days
using 32 cores on an ARM Neoverse-N1 processor, and 1 TB of DDR4 memory to check
approximately 140 billion tilings. Our implementation is available at
\url{https://github.com/maemre/triforce}. In addition, we provide pseudocode in \Cref{sec:pseudocode}.

We note that essentially the full size of the fixed region in \Cref{fig:inductive-gadget} is necessary in order for \Cref{lemma:inductive-gadget} to hold. Our experimentation found counterexample configurations on slightly smaller fixed regions. \Cref{fig:failing-cover} shows one such example:
The solid trapezoid is a slightly smaller fixed region, and the black and the red nodes together form a cover of this trapezoid. The partial tiling of red nodes shown in the figure is locked, and the only freedom of ReCom moves lies within the black nodes. (There are 4 possible ways to tile the black nodes and complete the partial tiling.) Adding the circled nodes at the bottom yields the fixed region in \Cref{fig:inductive-gadget}. Extending the tiling to cover these circled nodes would introduce more ReCom opportunities, so that we could put together a chain of ReCom moves to create two vertical tiles at the top.

\begin{figure}[h]
\centering
\includegraphics[width=4cm]{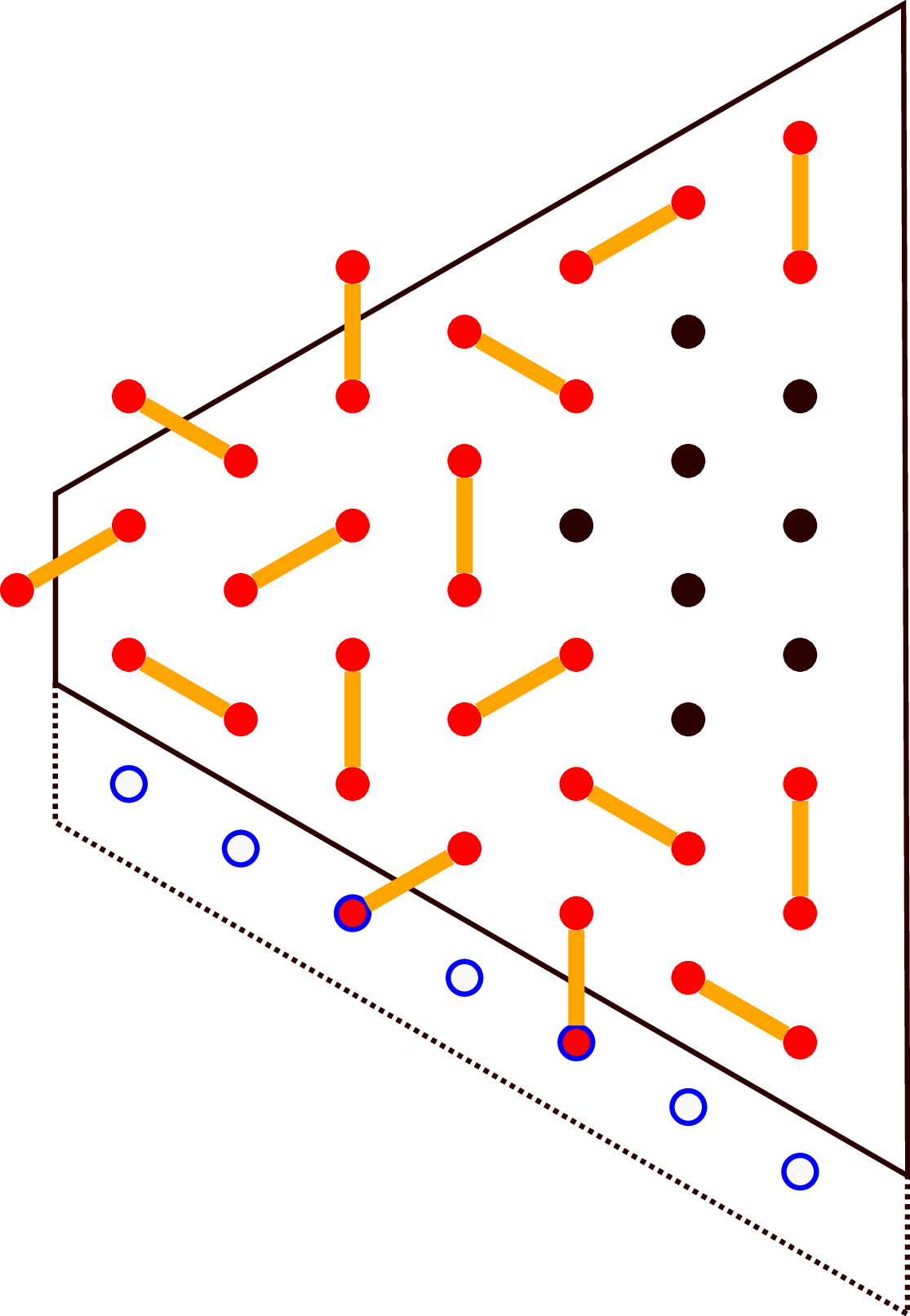}
\vspace{4pt}
\caption{A slightly smaller fixed region (the solid trapezoid) that has a cover with an almost-locked tiling. The given tiling of the red nodes is locked, so that any choice of tiles to cover the black nodes produces a counterexample to a modified version of \Cref{lemma:inductive-gadget} in which the solid trapezoid above replaces \Cref{fig:inductive-gadget} as the fixed region.  Adding the circled nodes to the solid trapezoid yields the region in \Cref{fig:inductive-gadget}.}
\label{fig:failing-cover}
\end{figure}

\subsection{Base cases: up to side-length 9} \label{sec:base-cases}

We use the computer program to directly check the metagraph of the \ctriangles
up to and including side-length 9. There are some irregularities for side-lengths 1--4:
\begin{itemize}
\item The \ctriangle of side-length 1 is just the empty graph, so its metagraph
  is vacuously connected.
\item The \ctriangle of side-length 2 is just a single 2-tile, so it has only
  one tiling. The metagraph is trivially connected.
\item The \ctriangle of side-length 3 has two tilings that are disconnected in
  the metagraph. These tilings are shown in \Cref{fig:triangle-3-metagraph}.
\item The metagraph of the \ctriangle of side-length 4 has three connected
  components that are rotations of each other. The metagraph is shown in
  \Cref{fig:triangle-4-metagraph}.
\end{itemize}

The metagraphs for side-lengths 5--9 are all connected.

\begin{figure}[h]
  \centering
  \includegraphics[width=2cm]{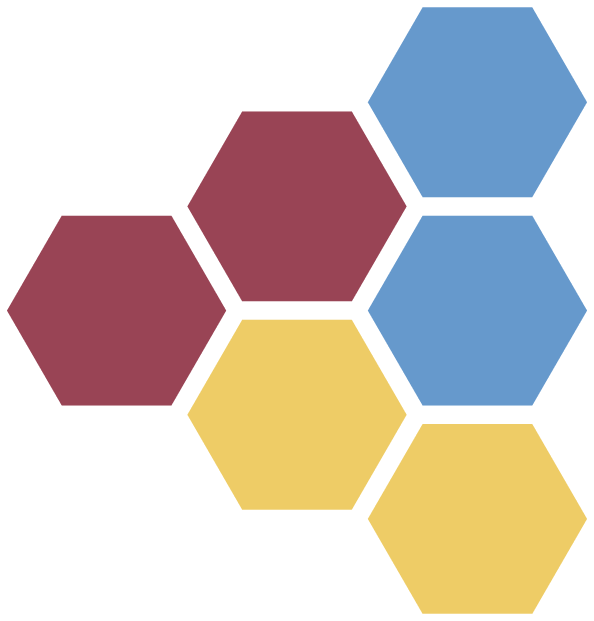}
  \hspace{2em}
  \includegraphics[width=2cm]{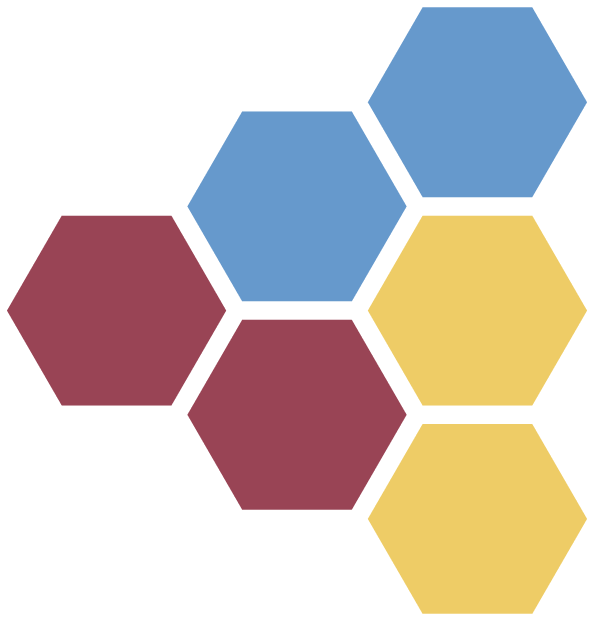}
  \caption{The metagraph for the triangle of side-length 3, also noted in \cite{cannon_irreducibility}. The two tilings are disconnected.}
  \label{fig:triangle-3-metagraph}
\end{figure}

\begin{figure}[h]
  \centering
\begin{tikzpicture}[
    pic/.style={inner sep=1pt, draw, rounded corners},
    node distance=8mm and 12mm
]
    \node[pic] (1a) {\includegraphics[width=2cm]{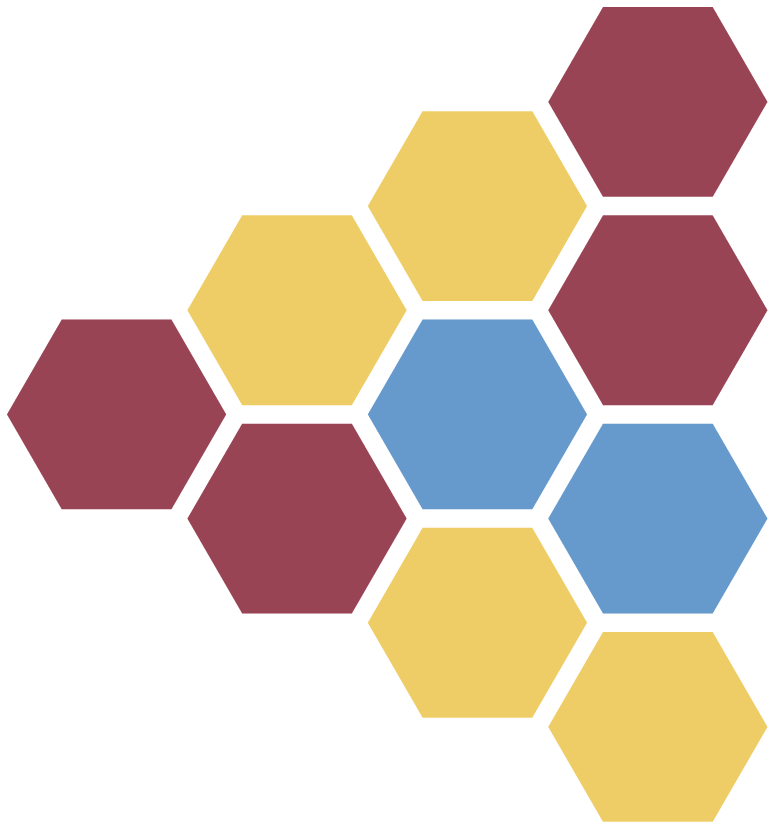}};
    \node[pic, below=of 1a] (1b) {\includegraphics[width=2cm]{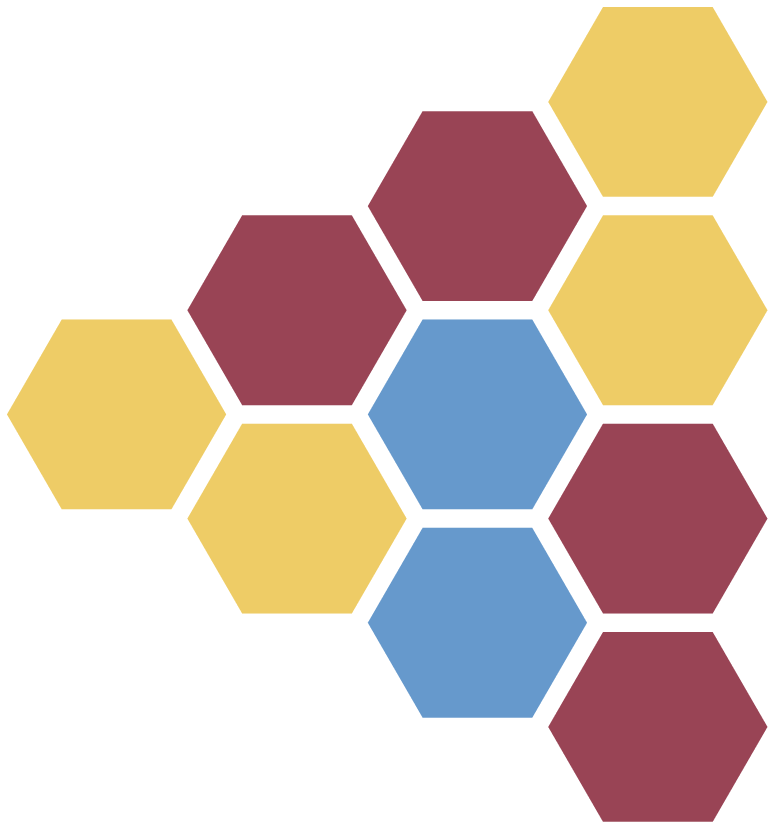}};

    \node[pic, right=of 1a, xshift=1cm] (2a) {\includegraphics[width=2cm]{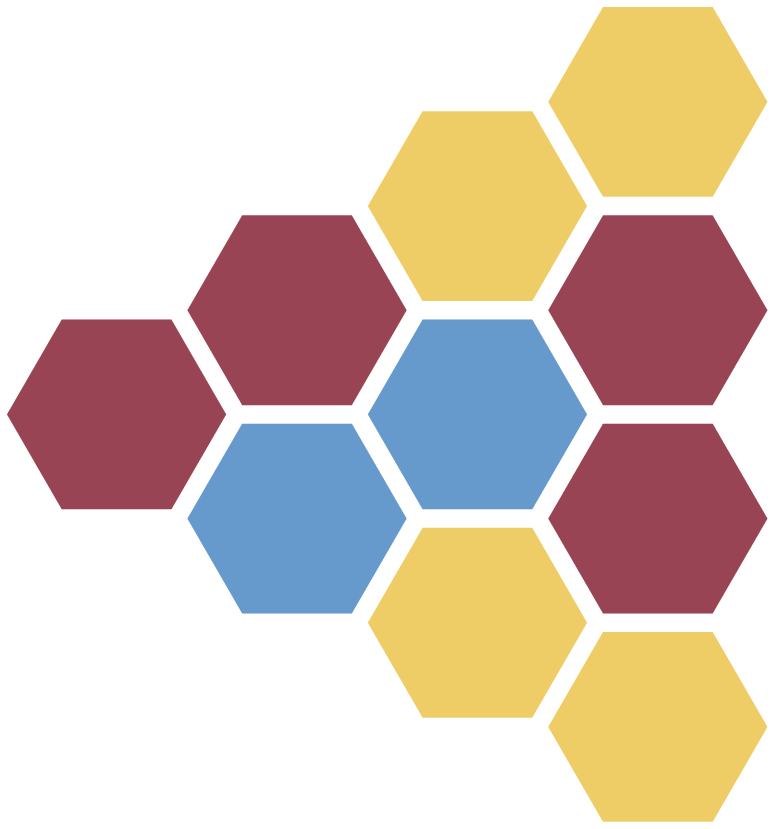}};
    \node[pic, below=of 2a] (2b) {\includegraphics[width=2cm]{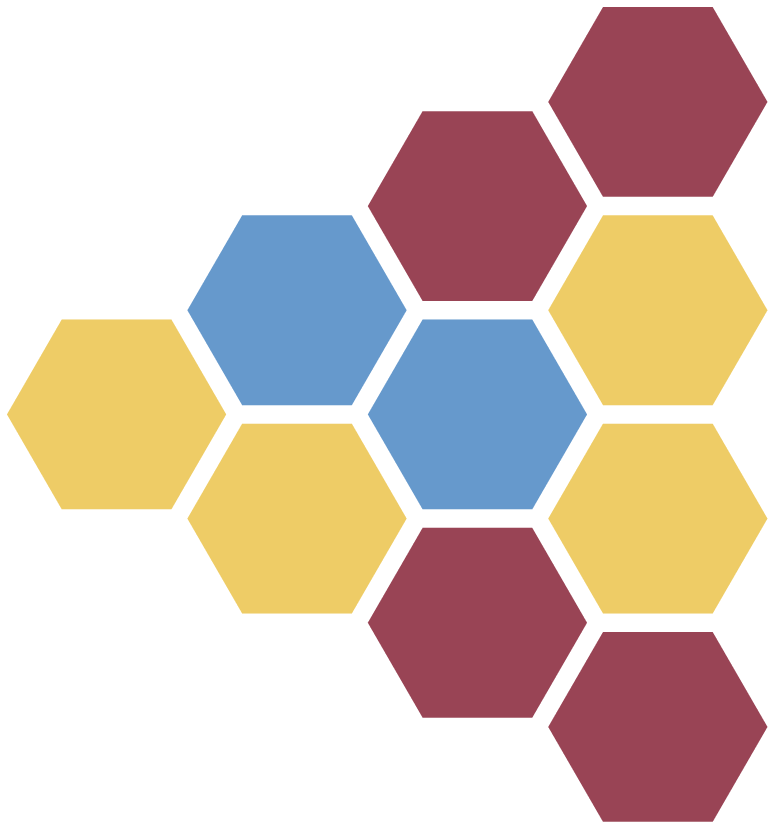}};

    \node[pic, right=of 2a, xshift=1cm] (3a) {\includegraphics[width=2cm]{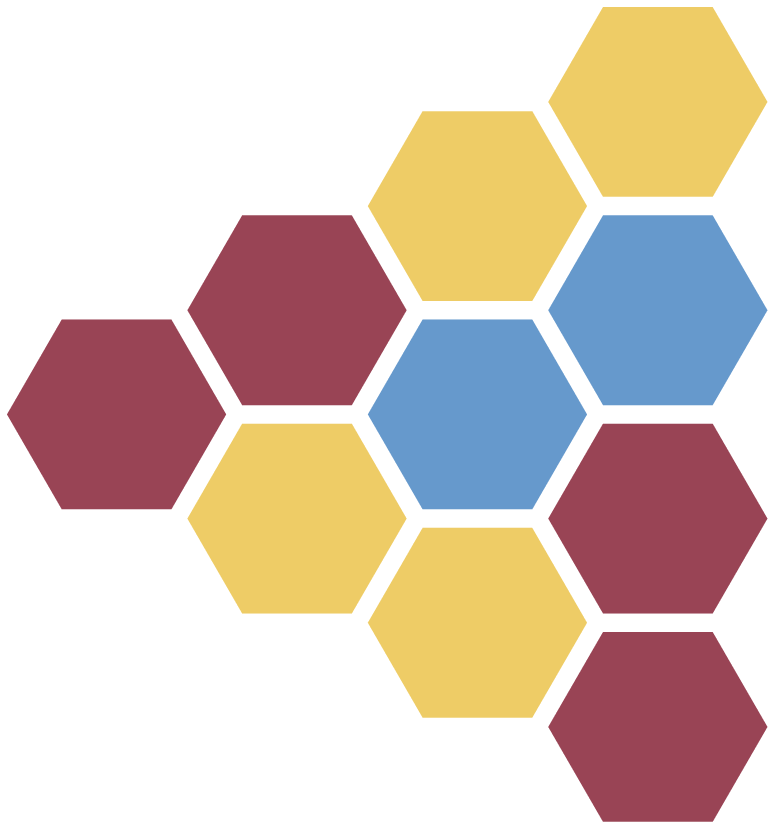}};
    \node[pic, below=of 3a] (3b) {\includegraphics[width=2cm]{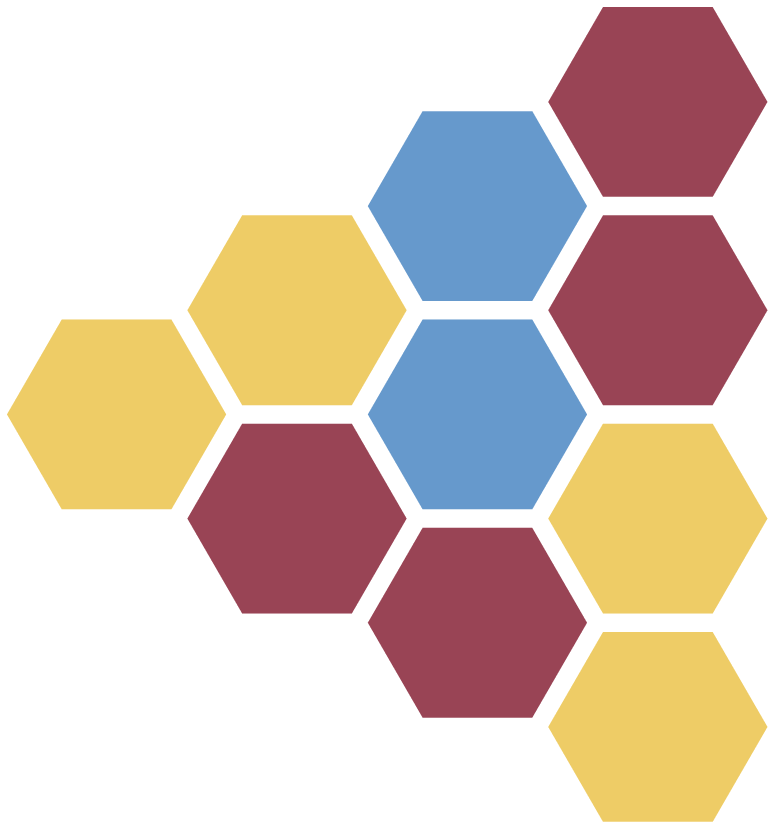}};

    \draw (1a) -- (1b);
    \draw (2a) -- (2b);
    \draw (3a) -- (3b);
\end{tikzpicture}
\vspace{8pt}
  \caption{The metagraph for the triangle of side-length 4. There are three connected components of size 2.}
  \label{fig:triangle-4-metagraph}
\end{figure}

\subsection{Proof of \Cref{thm:main_result_triangular_lattice}}
\label{sec:tile-2-inductive-proof}

As the base cases were verified in \Cref{sec:base-cases}, we consider an \ctriangle of side length $n>9$.

We prove that the dimer tiling metagraph on this \ctriangle is connected by starting from an arbitrary dimer tiling and applying ReCom moves to
reach a specific target tiling. The target tiling is described as follows: (1) the bottom side (from the left vertex to the bottom-right vertex) is made up of repeated copies of the same specific 4-triangle, namely the triangle of side-length 4 that is on
the left corner of the larger triangle in \Cref{fig:sample_target8}, and (2) the rest
of the triangle is tiled using only vertical tiles.  An example of this tiling
for the 8-triangle is in \Cref{fig:sample_target8}.

\begin{figure}[h]
\centering
  \includegraphics[width=3cm]{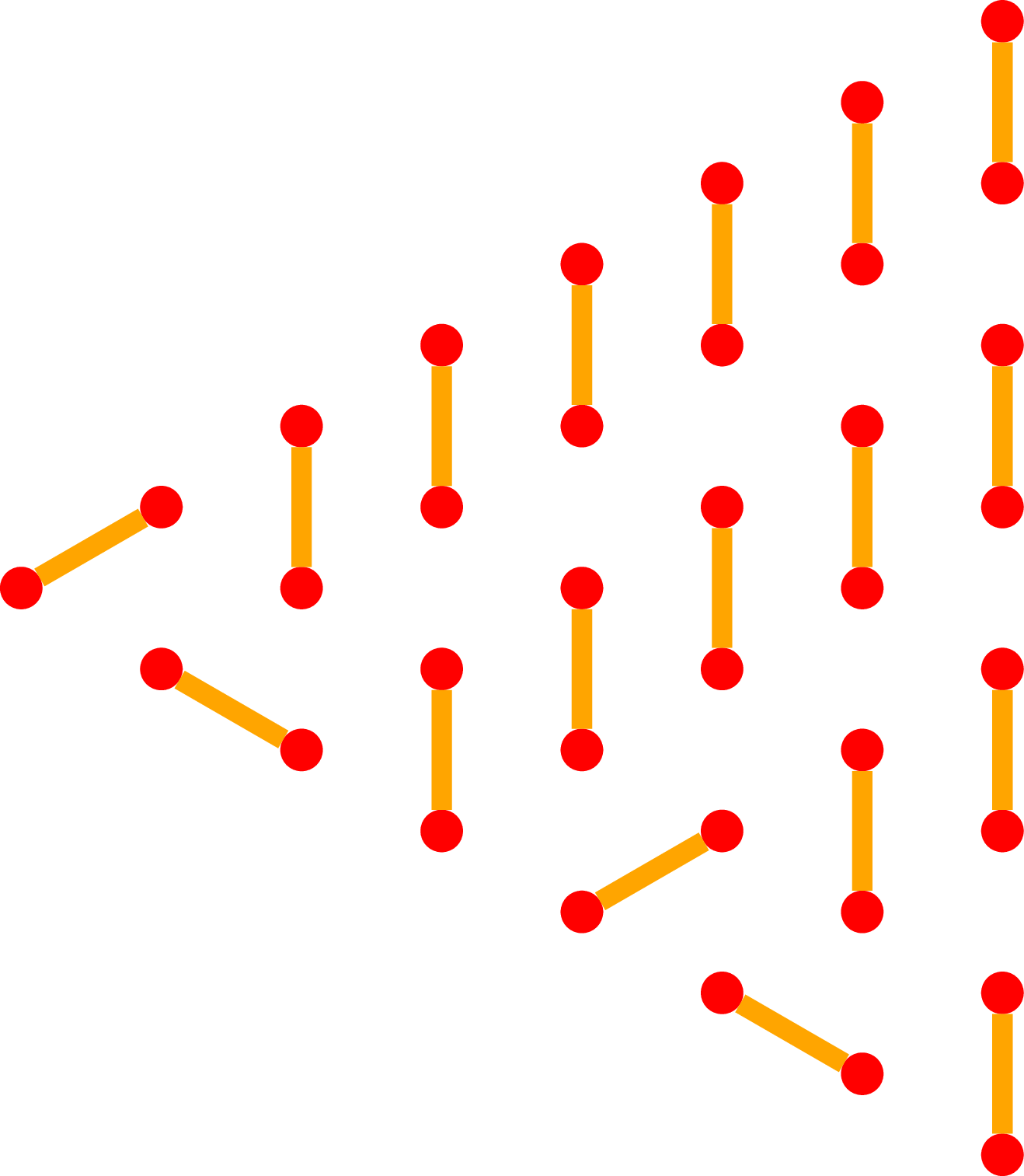}
  \vspace{8pt}
  \caption{Sample target tiling for $n = 8$.}
  \label{fig:sample_target8}
\end{figure}

This tiling can be extended to an arbitrary side length $n$ by (1) repeating the
4-triangle pattern at the bottom until we get to the next multiple of 4 after
$n$, (2) tiling the remaining cells using vertical tiles, and (3) removing the
right-hand columns until we get back down to $n$.  Notice that this process yields a full
triangle when $n \equiv 0, 3 \pmod 4$ and a clipped triangle when $n \equiv 1, 2 \pmod 4$.

Fix $n > 9$. By induction, and using the base cases from \Cref{sec:base-cases}, we may assume that the metagraph is connected
for an \ctriangle of side-length $n - 2$.  We will show that we can arrange the
\ctriangle of side-length $n$ to have two vertical columns on the right edge using
ReCom moves. Then, by the inductive hypothesis, the residual \ctriangle can be arranged into
the configuration we want to reach.

We build these two columns using a second inductive argument.  \Cref{lemma:inductive-gadget} states that
any covering of the red region in \Cref{fig:inductive-gadget} can be rearranged
to have two vertical tiles at the top corner.  This includes all the covers that
fit in the top corner of the large \ctriangle of side-length $n > 9$. Thus, by \Cref{lemma:inductive-gadget}, we can use ReCom moves to get two vertical tiles at the top corner of the large \ctriangle.

We next move the \Cref{fig:inductive-gadget} region downward by two cells in the large \ctriangle and use \Cref{lemma:inductive-gadget} again to get two more vertical tiles just below the first two. We keep sliding down the right edge two cells at a time, repeatedly using \Cref{lemma:inductive-gadget} to build two columns of vertical tiles. This continues until we get too close to the bottom for \Cref{fig:inductive-gadget} to fit. Then, we use \Cref{lemma:bottom-corner} to complete the two rightmost columns. They contain only vertical tiles except for a piece at the very bottom that matches the pattern shown in \Cref{fig:sample_target8}.

The remaining tiles cover a residual \ctriangle of side-length $n-2$, whose metagraph is connected by the inductive hypothesis. We make the appropriate ReCom moves so that the residual \ctriangle, and therefore the original \ctriangle of size $n$, matches the target tiling. \Cref{fig:sample_11_to_9} illustrates and further explains this step of the proof in the case $n=11$.
Since we started from an arbitrary tiling of the \ctriangle of size $n$ and reached the specific target, we conclude that the metagraph for
the side-length $n$ \ctriangle is connected.

\begin{figure}[h]
\centering
  \includegraphics[width=4cm]{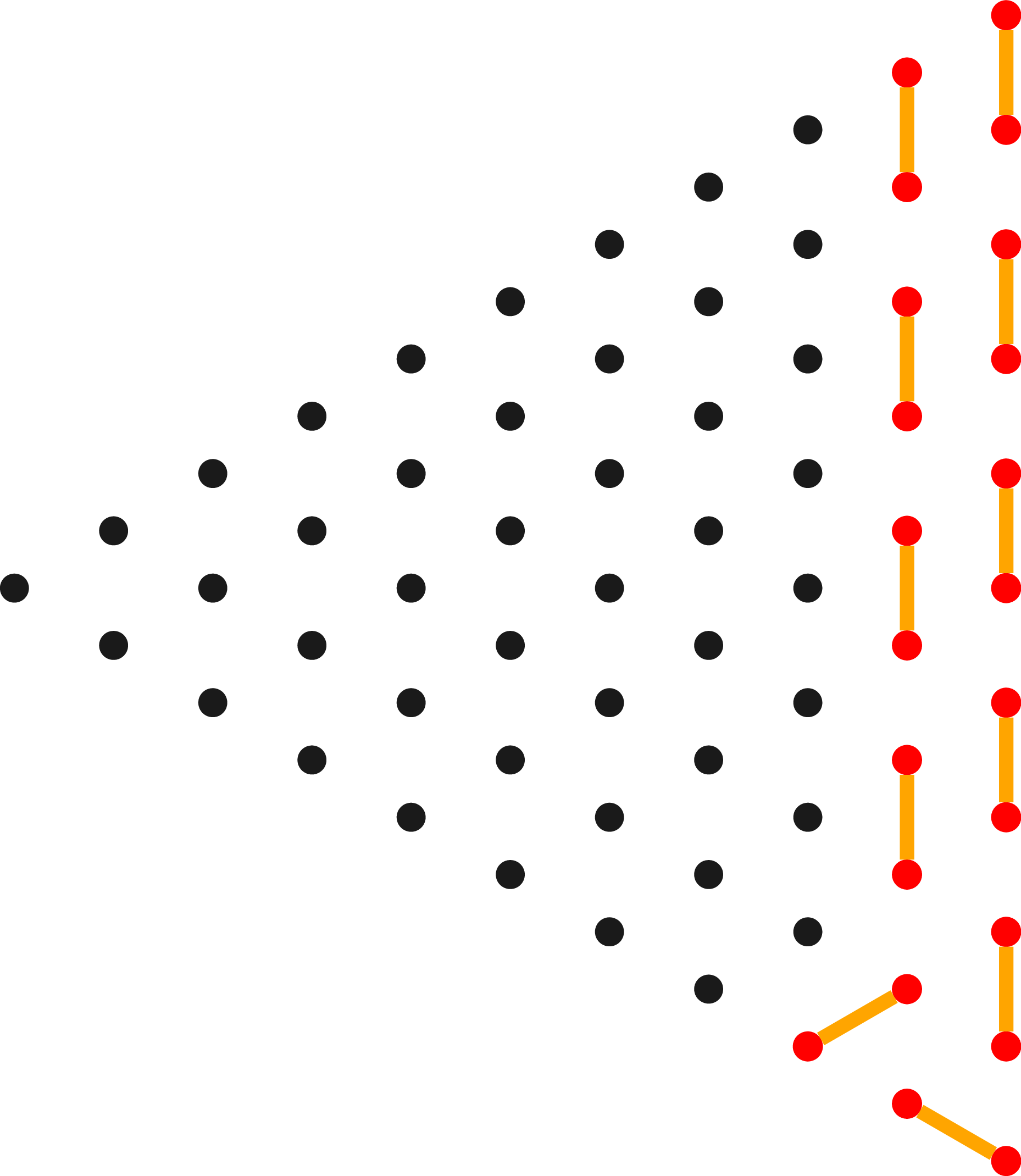}
  \vspace{8pt}
  \caption{A sample induction step in the case $n=11$. Starting with any dimer tiling of the \ctriangle of side-length $n$, the argument proceeds as follows. (1) Using \Cref{lemma:inductive-gadget,lemma:bottom-corner}, perform ReCom moves to get the orange tiles shown. (2) The remaining tiles cover the black nodes, which are precisely an \ctriangle of side-length $n-2$. This residual \ctriangle has a connected metagraph by the inductive hypothesis. Perform ReCom moves on the residual \ctriangle so that its tiles match the target tiling. (3) Looking together at the residual \ctriangle and the orange tiles shown, the entire tiling on the side-length $n$ \ctriangle matches the target tiling.}
  \label{fig:sample_11_to_9}
\end{figure}

\begin{figure}[h]
\centering
  \begin{subfigure}[b]{0.45\textwidth}
  \includegraphics[width=4cm]{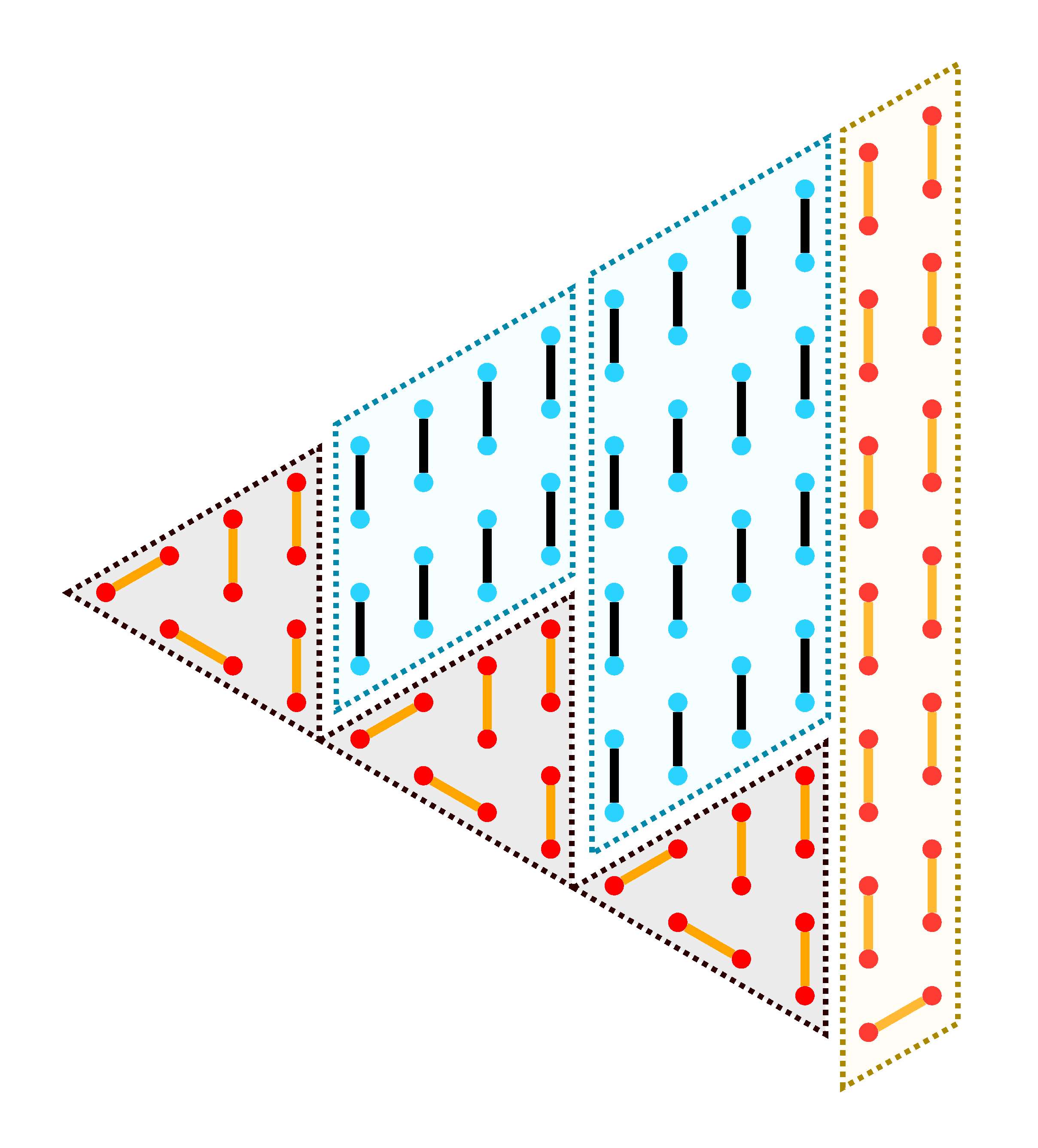}
  \caption{}
  \label{fig:target-14}
  \end{subfigure}
  \hfill
  \begin{subfigure}[b]{0.45\textwidth}
  \includegraphics[width=4cm]{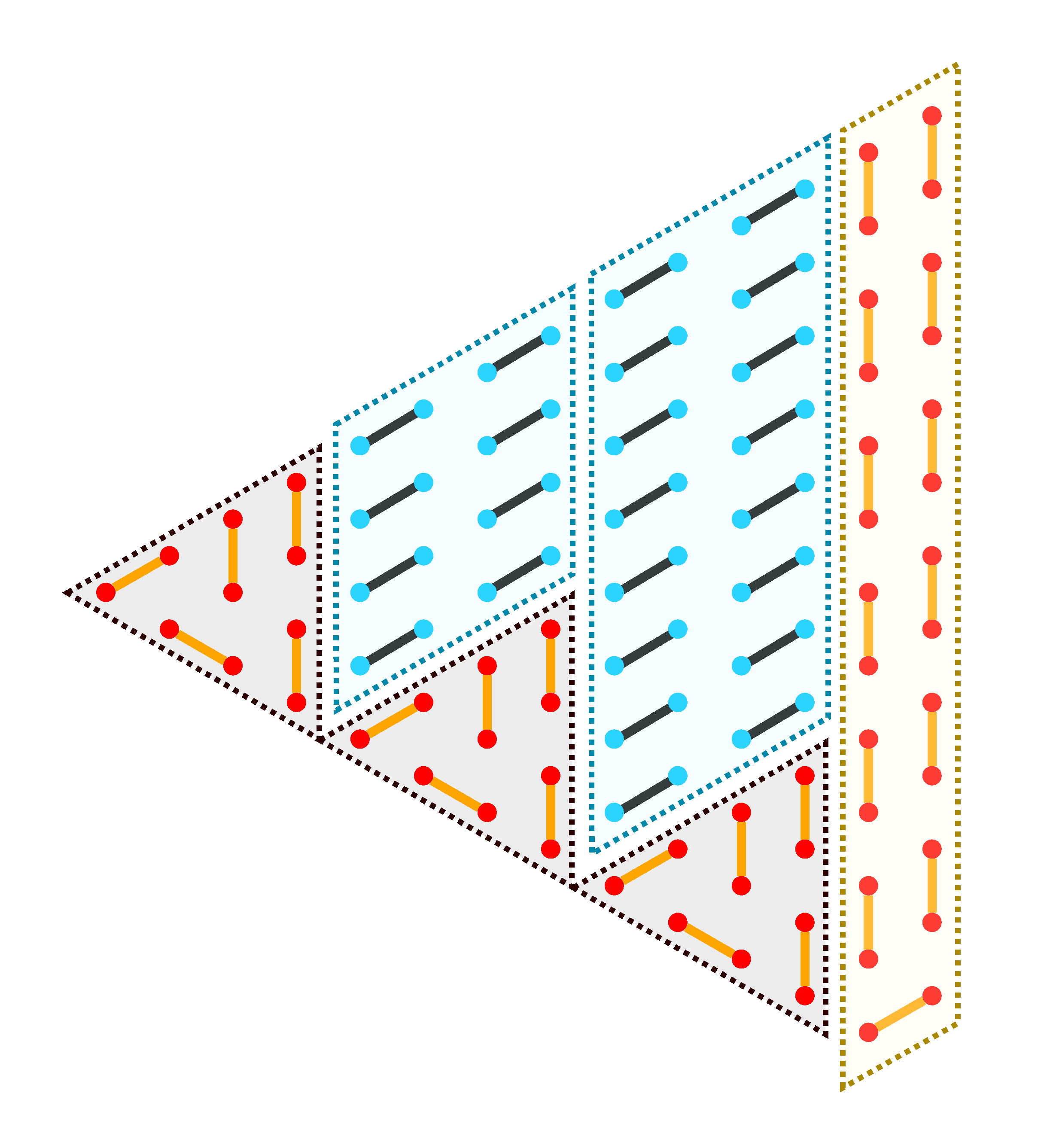}
  \caption{}
  \label{fig:other-14}
  \end{subfigure}
  \caption{Two sample tilings of the \ctriangle with side length $n = 14$. Tiling (a) is the target of the inductive method in our proof, and tiling (b) differs from it in all the tiles inside the parallelograms.  These two tilings are at distance $\Omega(n^2)$ from each other in the metagraph.}
  \label{fig:diameter-lower-bound-example}
\end{figure}

To show that the metagraph's diameter is in $\Theta(n^2)$, we separately show that it is in $O(n^2)$ and $\Omega(n^2)$:
\begin{itemize}
    \item \emph{The diameter is in $O(n^2)$:} The argument above shows how to get from any tiling of the side-length $n$ \ctriangle to the target tiling (e.g. in \Cref{fig:sample_target8}) via ReCom moves. To arrange the two rightmost columns of the \ctriangle into vertical tiles (with the bottom matching the pattern in \Cref{fig:sample_target8}) requires $O(n)$ invocations of \Cref{lemma:inductive-gadget} and one invocation of \Cref{lemma:bottom-corner}. Each invocation of \Cref{lemma:inductive-gadget} or \Cref{lemma:bottom-corner} uses $O(1)$ ReCom moves. Thus, it takes $O(n)$ ReCom moves to construct the two rightmost columns of the target tiling. Proceeding by induction, it takes $O(n^2)$ ReCom moves to reach the target tiling on the full \ctriangle.
    
    \item \emph{The diameter is in $\Omega(n^2)$:} Below, we exhibit a pair of tilings that differ by $\Omega(n^2)$ tiles. Since each ReCom move affects only two tiles, it takes $\Omega(n^2)$ moves to get from one tiling to the other. This shows that the diameter is in $\Omega(n^2)$.

    \Cref{fig:diameter-lower-bound-example} shows the two tilings in the case $n=14$. The first tiling, \Cref{fig:target-14}, is the target tiling that we previously used when proving that the metagraph is connected. It is divided into three regions: the repeating pattern of triangles at the bottom (shaded gray), the vertical tiles lying above those triangles (shaded blue), and the extra tiles at the right (shaded yellow). The blue region of vertical tiles is subdivided into parallelograms of width 4 and increasing heights.

    The second tiling, \Cref{fig:other-14}, is the same as \Cref{fig:target-14} except that all the tiles in the blue parallelograms have been rotated so that they are parallel to the top and bottom sides of the parallelogram. To find the number of differing tiles between \Cref{fig:target-14,fig:other-14}, we count the total number of tiles in the parallelograms. There are $\lfloor n/4 \rfloor - 1$ parallelograms, and the $i$-th parallelogram contains $8i$ tiles. Thus, the total number of tiles in the parallelograms is
    \[
        \sum_{i=1}^{\lfloor n/4 \rfloor - 1} 8i = 8 \cdot \frac{(\lfloor n/4 \rfloor - 1) \lfloor n/4 \rfloor}{2} \in \Omega(n^2).
    \]
    We conclude that the two tilings differ by $\Omega(n^2)$ tiles, so the distance between them in the metagraph is $\Omega(n^2)$. \qed
\end{itemize}

\section{Locked configurations with few large districts}\label{sec:locked_few_districts}

In this section, we construct an infinite family of six-district locked configurations on dual graphs that are triangulations. \Cref{fig:locked-6} shows the result of the construction.

\begin{Theorem}\label{thm:locked-6}
There is an infinite family of dual graphs $(G_k)_{k \geq 6}$ with the following properties:
\begin{itemize}
\item $G_k$ is a triangulation with maximum degree 8.
\item $G_k$ has $12k^2$ nodes.
\item $G_k$ has a locked configuration of 6 districts with $2k^2$ nodes each.
\end{itemize}
\end{Theorem}

\begin{proof}

Let $k \geq 6$. Our construction is illustrated by \Cref{fig:locked-6}, which shows the cases $k=8$ and $k=17$. We use the following procedure:
\begin{enumerate}
    \item The dual graph is drawn as a regular triangular lattice in which each node is represented by a hexagonal cell.
    \item The cells are colored using six colors.
    \item Certain extra nodes are added to the dual graph as shown in \Cref{fig:hexagon-add}.
\end{enumerate}  

In a nutshell, Steps 1 and 2 formalize the winding pattern of districts in \Cref{fig:locked-6}. They amount to routine bookkeeping. Step 3, which shows how and where to add extra \Cref{fig:hexagon-add} nodes to obtain a locked configuration, contains the crux of the argument.

\begin{figure}[h]
\centering
\includegraphics[width=0.2\textwidth]{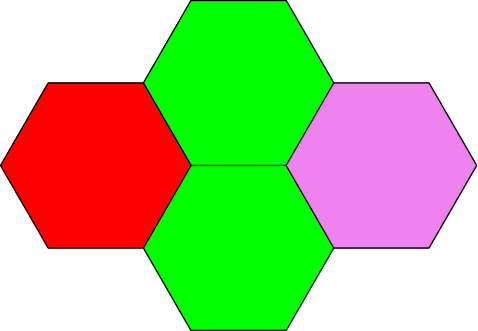} \qquad \includegraphics[width=0.2\textwidth]{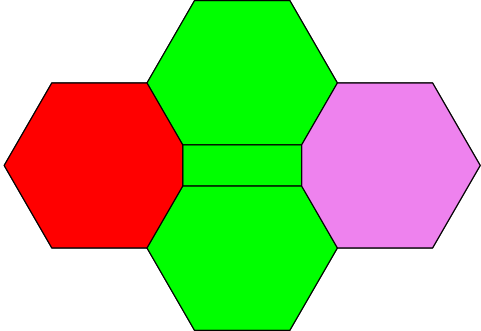}
\vspace{8pt}
\caption{Adding a node to the dual graph.}
\label{fig:hexagon-add}
\end{figure}

\textbf{Step 1: Draw the dual graph as a regular triangular lattice.} The dual graph begins as a regular triangular lattice whose nodes are drawn as hexagonal cells. These cells are arranged in the shape of a large hexagon. Then, three parallelogram shapes are removed from the corners of the large hexagon. \Cref{fig:hex-cutoff} shows the large hexagon in the case $k=8$. The cells making up the three parallelogram shapes, which will be removed, are shaded in gray.

\begin{figure}[t]
\centering
\includegraphics[width=0.7\textwidth]{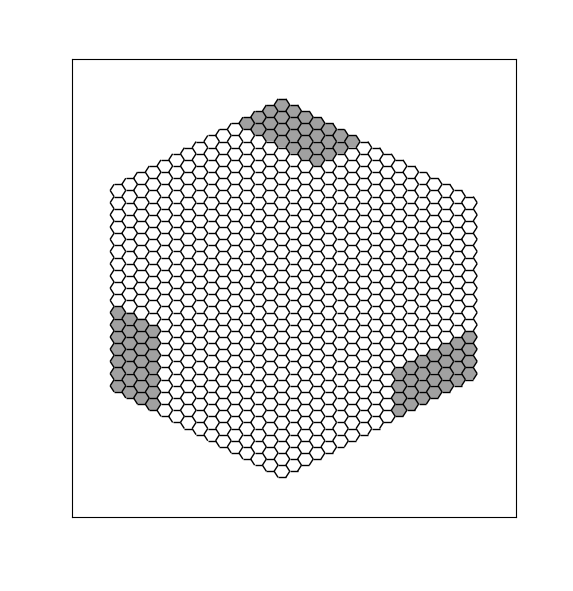}
\caption{First steps of drawing the dual graph for \Cref{thm:locked-6} in the case $k=8$. The side lengths of the large hexagon, going clockwise from the Northwest vertex, are: $2k-1$, $2k+1$, $2k-1$, $2k+1$, $2k-1$, $2k+1$. The gray cells make up the three $4 \times (k-1)$ parallelograms which are removed from the North, Southeast, and Southwest corners.}
\label{fig:hex-cutoff}
\end{figure}

Before the three parallelogram shapes are removed, the large hexagon has six vertices which we label as Northwest, North, Northeast, Southeast, South, and Southwest. (Throughout the argument, we will use compass directions for orientation.) The large hexagon is not a regular hexagon.  Starting from the Northwest vertex and going clockwise, the Northwest--North side has $2k-1$ cells, then the North--Northeast side has $2k+1$ cells, then the side lengths continue to alternate between $2k-1$ and $2k+1$.  We draw a $4 \times (k-1)$ parallelogram at the North vertex, containing the last 4 cells of the Northwest--North side and the first $k-1$ cells of the North--Northeast side, and remove it from the graph. Finally, we enforce $120^\circ$ rotational symmetry by removing similar $4 \times (k-1)$ parallelograms from the Southeast and Southwest vertices. See \Cref{fig:hex-cutoff}.

As we will see in Step 3 of the proof, the construction allows for almost complete freedom in choosing the shape of the outer boundary. In particular, we could avoid removing the three $4 \times (k-1)$ parallelograms at the cost of increasing the number of extra \Cref{fig:hexagon-add} nodes added during Step 3. This would also increase the minimum value of $k$ from 6 to 10.

\textbf{Step 2: Color the cells.} In both the top and bottom panels of \Cref{fig:locked-6}, the center of the dual graph is the meeting point of the red, green, and dark blue districts. This point is marked with a large black dot in \Cref{fig:hexagon-middle}. At this stage, the cells have not yet been colored and the extra \Cref{fig:hexagon-add} nodes have not yet been added.

\begin{figure}[h]
\centering
\includegraphics[width=0.3\textwidth]{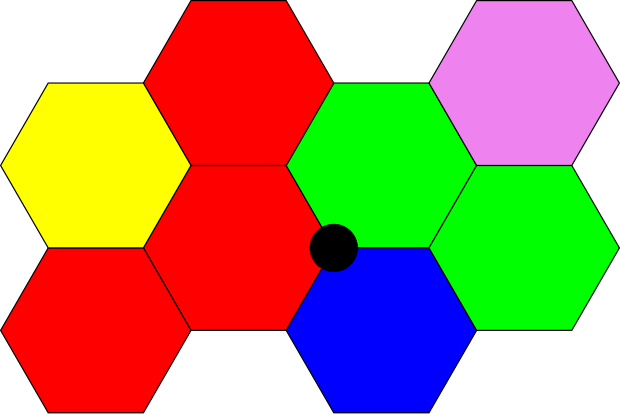}
\vspace{8pt}
\caption{Center of the locked configurations shown in \Cref{fig:locked-6}, before the extra nodes are added to the dual graph.}
\label{fig:hexagon-middle}
\end{figure}

Start with the cell immediately to the West of the black dot. Color it red. Looking at the top panel of \Cref{fig:locked-6}, the red district extends to the North and Southwest from this ``anchor'' cell. Let us focus on the North direction and ignore the extra \Cref{fig:hexagon-add} nodes, which are not yet part of the graph. The red district goes forward for 6 cells, then turns left, then goes forward for another 6 cells, then left again, then forward 5, left, forward 5, and so forth. We describe this path using the notation $6\LL 6\LL 5\LL 5\LL\cdots 2\LL 2\LL 1\LL 1$. For general $k \geq 6$, the red district going North from the anchor cell follows the path
\[
(k-2)\LL(k-2)\LL(k-3)\LL(k-3)\LL\cdots 2\LL 2\LL 1\LL 1.
\]
The red district going Southwest from the anchor cell follows the path
\[
(k-2)\LL(k-1)\LL(k-2)\LL(k-2)\LL(k-3)\LL(k-3)\LL\cdots 2\LL 2\LL 1\LL 1.
\]
At this point, the total number of cells in the red district is
\[
1 + \left( \sum_{j=1}^{k-2} 2j \right) + \left( (k-2) + (k-1) + \sum_{j=1}^{k-2} 2j \right) = 2k^2 - 4k + 2,
\]
where the initial 1 comes from the anchor cell itself, the first bracket comes from the North direction, and the second bracket comes from the Southwest direction. This is $4k-2$ cells short of the target of $2k^2$. We will eventually add $4k-2$ extra red cells as in \Cref{fig:hexagon-add}.

The red district is now finished except for the extra \Cref{fig:hexagon-add} cells. Rotating $120^\circ$ clockwise about the black dot gives the green district, and rotating another $120^\circ$ clockwise about the black dot gives the dark blue district.

We now color the yellow district. Start by coloring a yellow anchor cell which is immediately to the Northwest of the red anchor cell. (This is the yellow cell in \Cref{fig:hexagon-middle}.) From this anchor, the yellow district extends to the North and Southwest. The path to the North is
\[
(k-3)\LL(k-3)\LL(k-4)\LL(k-4)\LL\cdots 2\LL 2\LL 1\LL 1.
\]
The path to the Southwest, using R for right turns, starts with
\[
(k-2)\RR(k-2)\RR(k-1)\RR(k-1)\RR(k-1)\LL(k-1)\RR(k)\RR (k-1).
\]
Then, the yellow district opens out into a parallelogram of size $(k-5) \times k$ and finally adds a tail of 4 cells at the end. Incidentally, the $(k-5)$ here is one reason for the requirement that $k \geq 6$.

At this point, the total number of cells in the yellow district is
\[
1 + \left( \sum_{j=1}^{k-3} 2j \right) + \Bigg( 2(k-2) + 4(k-1) + k + (k-1) + (k-5)k + 4 \Bigg) = 2k^2 - 2k + 2.
\]
We will eventually add $2k-2$ extra yellow cells to reach the target of $2k^2$.

The yellow district is now finished except for these extra cells. Rotating $120^\circ$ clockwise about the black dot gives the purple district, and rotating another $120^\circ$ clockwise about the black dot gives the light blue district.

This completes the cell coloring. We omit the (tedious) proof that each cell has received one and only one color according to this procedure, appealing instead to the visuals of \Cref{fig:locked-6}.

\textbf{Step 3: Add extra nodes to the dual graph.} In this step, we add extra nodes as in \Cref{fig:hexagon-add} so that all six districts have equal size and so that the configuration is locked. This could also be accomplished by choosing some nodes to have population 2 instead of population 1, as in \cite{JTF_tilings}. That would preserve the strict triangular lattice structure of the dual graph. We choose instead to maintain the requirement that all nodes have population 1 and balance the districts by adding extra nodes, turning the dual graph into a slightly irregular triangulation of bounded degree.

In the end, it doesn't much matter which procedure we use to balance the districts. We will set up a system of inequalities that must be satisfied in order for the districts to have equal size and for the configuration to be locked. Importantly, the system has a solution, but only barely. This means that the configurations shown in \Cref{fig:locked-6} are locked, but they rely on the exact number and location of the extra nodes. It is reasonable to conjecture that a ``generic'' or random bounded degree triangulation should not admit any locked configurations with high probability.

We will add extra nodes in a way that preserves the $120^\circ$ rotational symmetry of the figure. Because of this, we need only verify the following conditions: (i) the red/green/dark blue districts have the same number of nodes as the yellow/purple/light blue districts; (ii) the red and green districts are locked under ReCom; (iii) the red and yellow districts are locked under ReCom; (iv) the red and light blue districts are locked under ReCom; and (v) the yellow and light blue districts are locked under ReCom.

Conditions (iii), (iv), and (v) signal the importance of the triple point where the red, yellow, and light blue districts meet. See the upper panel of \Cref{fig:locked-6}. We identify the three cells touching this point as the red triple cell, yellow triple cell, and light blue triple cell.

The red district extends from the red triple cell in two paths to the South and to the Northeast. Suppose that after adding extra nodes, the Southern red path has $r_1$ cells and the Northeastern red path has $r_2$ cells. We do not include the red triple cell itself in either of these paths. Thus, the total number of cells in the red district is $m = 1+r_1+r_2$. For condition (i) to hold, the yellow and light blue districts must also have $m$ cells each.

The yellow district extends from the yellow triple cell in two paths to the Northeast and to the Northwest. Suppose that after adding extra nodes, the Northeastern yellow path has $y_1$ cells and the Northwestern yellow path has $y_2$ cells. The total number of cells in the yellow district is $m = 1+y_1+y_2$.

The light blue district extends from the light blue triple cell in two paths to the Northwest and to the South. Suppose that after adding extra nodes, the Northwestern light blue path has $b_1$ cells and the Southern light blue path has $b_2$ cells. The total number of cells in the light blue district is $m = 1+b_1+b_2$.

Let us consider condition (iii). When the red and yellow districts are combined, they form a segment of width 2 that splits into two tails of width 1. The red and yellow districts will be locked under ReCom only if the combination of the two tails has at least $m+1$ cells. From our definitions above, the red tail has $1+r_1$ cells and the yellow tail has $1+y_2$ cells. Thus, $(1+r_1) + (1+y_2) \geq m+1$, so $1 + r_1 + y_2 \geq m$. As $1 + r_1 + r_2 = m$ and $1 + y_1 + y_2 = m$, we obtain the inequalities $r_1 \geq y_1$ and $y_2 \geq r_2$.

Condition (iv) about the red and light blue districts yields the similar inequality $(1+b_1) + (1+r_2) \geq m+1$, so $1+b_1+r_2 \geq m$. This implies that $b_1 \geq r_1$ and $r_2 \geq b_2$. Likewise, condition (v) implies that $y_1 \geq b_1$ and $b_2 \geq y_2$.

We now have the inequalities $r_1 \geq y_1 \geq b_1 \geq r_1$ and $r_2 \geq b_2 \geq y_2 \geq r_2$. This means that $r_1 = y_1 = b_1$ and $r_2 = y_2 = b_2$.

Conversely, if $r_1 = y_1 = b_1$ and $r_2 = y_2 = b_2$, then condition (i) certainly holds. Conditions (iii) and (iv) are satisfied since in both cases, the combination of the two tails has exactly $m+1$ cells.

Condition (v), about the yellow and light blue districts, is more subtle because of the solid light blue parallelogram. Consider the union of the yellow and light blue districts in either panel of \Cref{fig:locked-6}. Since $r_1 = y_1 = b_1$ and $r_2 = y_2 = b_2$, the combination of the two tails has exactly $m+1$ cells. When partitioning this union into two connected sets of $m$ cells each, we may assume that the yellow tail remains yellow and the light blue tail remains light blue. The two districts keep their original colors until reaching the first parallelogram, which was originally light blue.

Consider now the second parallelogram, which was originally yellow. Since it can only be reached through a passage of width 1, it must be all the same color. Based on the cells which are already colored light blue, the light blue district only has enough cells left for one parallelogram plus a tail of 4 cells. If we assigned the second parallelogram to the light blue district, it would use up all its cells while still being disconnected. So, the second parallelogram must remain yellow.

We know that the yellow district must extend all the way to the second parallelogram. Since it can't use any more cells than it originally had, it must transit the first parallelogram by a direct path of minimum length. There are lots of minimum length paths through the parallelogram, but they all lead to a disconnected blue district except for the path that hugs the bottom boundary. In other words, the yellow district must keep all of its original yellow cells. This proves condition (v), that the yellow and light blue districts are locked under ReCom.

Finally, we consider condition (ii). The same triple point analysis as before provides a necessary and sufficient condition for the red/green, green/dark blue, and red/dark blue pairs of districts to be locked under ReCom. But, this condition is automatically satisfied due to the $120^\circ$ rotational symmetry of these three districts. In other words, if we add extra nodes in a rotationally symmetric fashion, condition (ii) is guaranteed.

We conclude that if rotational symmetry is preserved, the conditions $r_1 = y_1 = b_1$ and $r_2 = y_2 = b_2$ are necessary and sufficient to verify all five conditions (i)-(v) and obtain a valid locked configuration.

Now it is time for some computations. Before adding any extra \Cref{fig:hexagon-add} nodes, here are the cell counts for the six paths extending from the red/yellow/light blue triple point. (We use standard abbreviations for the compass directions.)
\begin{align*}
\text{S red path:} && (k-1) + \sum_{j=1}^{k-2} 2j &= k^2 - 2k + 1 \\
\text{NE red path:} && (k-2) + \sum_{j=1}^{k-2} 2j &= k^2 - 2k \\
\text{NE yellow path:} && (k-2) + \sum_{j=1}^{k-3} 2j &= k^2 - 4k + 4 \\
\text{NW yellow path:} && (k-2) + 4(k-1) + k + (k-1) + (k-5)k + 4 &= k^2 + 2k - 3 \\
\text{NW light blue path:} && (k-1) + k + (k-1) + (k-5)k + 4 &= k^2 - 2k + 2 \\
\text{S light blue path:} && (k-1) + \sum_{j=1}^{k-1} 2j &= k^2 - 1
\end{align*}

If we add $c_1$ extra nodes to the Southern red path and $c_2$ extra nodes to the Northeastern red path, we will have
\begin{align*}
r_1 &= (k^2-2k+1) + c_1 \\
r_2 &= (k^2-2k) + c_2.
\end{align*}

The yellow and light blue districts require more care, since we must add nodes in a way that respects the rotational symmetry. We split these districts into three parts: the inner spiral, the middle path, and the outer end. Adding an extra node to the inner spiral contributes to the Northeastern yellow path and the Southern light blue path. Adding an extra node to the middle path contributes to the Northwestern yellow path and the Southern light blue path. Adding an extra node to the outer end contributes to the Northwestern yellow path and the Northwestern light blue path.

If we add $d_1$ extra nodes to the inner spiral, $d_2$ extra nodes to the middle path, and $d_3$ extra nodes to the outer end, we will have
\begin{align*}
y_1 &= (k^2-4k+4) + d_1 \\
y_2 &= (k^2+2k-3) + (d_2+d_3) \\
b_1 &= (k^2-2k+2) + d_3 \\
b_2 &= (k^2-1) + (d_1+d_2).
\end{align*}

To ensure that $r_1=y_1=b_1$ and $r_2=y_2=b_2$, we must choose $c_1,c_2$ and $d_1,d_2,d_3$ so that
\begin{equation} \label{eq:triple-condition-1}
(k^2-2k+1) + c_1 = (k^2-4k+4) + d_1 = (k^2-2k+2) + d_3
\end{equation}
and
\begin{equation} \label{eq:triple-condition-2}
(k^2-2k) + c_2 = (k^2+2k-3) + (d_2+d_3) = (k^2-1) + (d_1+d_2).
\end{equation}
Both the second equality in \eqref{eq:triple-condition-1} and the second equality in \eqref{eq:triple-condition-2} reduce to the single equation
\[
d_1-d_3 = 2k-2.
\]
So, we can choose any nonnegative values for $d_2$ and $d_3$. Then, set $d_1 = d_3 + 2k-2$ and use \eqref{eq:triple-condition-1} and \eqref{eq:triple-condition-2} to find the required values of $c_1,c_2$.

For the construction in \Cref{fig:locked-6}, we set $d_2 = d_3 = 0$. This minimizes the number of extra nodes. We obtain
\begin{align*}
d_1 &= 2k-2 \\
c_1 &= 1 \\
c_2 &= 4k-3.
\end{align*}
Thus, we must add $2k-2$ extra nodes to the yellow inner spiral (which is the Northeastern yellow path), 1 extra node to the Southern red path, and $4k-3$ extra nodes to the Northeastern red path. Since $k \geq 6$, the Northeastern yellow path with its $k^2-4k+4$ cells has room for $2k-2$ extra nodes, and the Northeastern red path with its $k^2-2k$ cells has room for $4k-3$ extra nodes.

These additions are copied over to the purple and light blue districts as well as the green and dark blue districts to maintain rotational symmetry. \Cref{fig:locked-6} shows how to accomplish this while controlling the maximum degree of the triangulation. In the particular construction displayed, the maximum degree is 8. This could be reduced to 7 with a few minor changes.

Finally, the total number of nodes in each district is
\[
\begin{split}
m &= 1 + r_1 + r_2 = 1 + y_1 + y_2 = 1 + b_1 + b_2 \\
&= 1 + (k^2-2k+2) + (k^2+2k-3) = 2k^2.
\end{split}
\]
This completes the construction.
\end{proof}

Suppose that we want to change the shape of the outer boundary. This can easily be done by increasing $d_3$ by an amount that is linear in $k$ and adjusting the values of $d_1,c_1,c_2$ accordingly. Since the changes are linear in $k$, there will be enough room to add the required extra nodes along the appropriate paths as long as $k$ is sufficiently large. The extra nodes in the outer ends of the yellow, purple, and light blue districts can be molded into whatever boundary shape is desired. The configuration will remain locked as long as there is still a ``choke point'' of width 1 in the combined yellow/light blue district and its rotations. We note that despite the freedom in choosing the outer boundary shape, the existence of a locked configuration still relies on choosing the number and location of the extra nodes very carefully.

\section{Districting the grid with 3 districts}\label{sec:thin_rectangle}
Tucker-Foltz explored many metagraphs on the rectangular grid \cite{JTF_tilings} with tilings of a fixed size, and a family of weighted grid graphs (each node having population either 1 or 2) with exactly 6 districts.    We\footnote{Anecdotally, we know that Tucker-Foltz and others have also explored the case of three districts and a generic rectangular grid.  To our knowledge, no results have yet been proved.} considered the case of three districts in the generically-sized rectangular grid.  In this case, 2/3 of the nodes can be re-combined and re-split in each ReCom step, so it seems reasonable to expect that the corresponding metagraph is connected.  The main complication seems to be the kind of winding that districts can do.  

Indeed, many of Tucker-Foltz's examples of locked tilings have districts which wind around each other \cite{JTF_tilings}.  Our own examples in \Cref{fig:locked-6} have significant winding of districts.  And a key Lemma in Cannon's results on the triangular graph using three districts is called the ``Unwinding Lemma'' \cite{cannon_irreducibility}.

\Cref{thm:3_by_n_rectangle} avoids the winding issue simply because the rectangular region is too narrow to have connected districts with significant winding around each other:

\begin{restatable}{Theorem}{rectanglethm}\label{thm:3_by_n_rectangle}
    Suppose $G$ is the subgraph of the grid graph which is induced by nodes in a $3\times n$ rectangle.  We consider districting maps which consist of 3 districts, each of size $n$.  Then the corresponding metagraph resulting from ReCom moves on this districting map is connected.
\end{restatable}

We note that \Cref{thm:3_by_n_rectangle} would follow from the main result in \cite{Akitaya_et_al} if we allowed the district sizes to vary by $\pm 1$ when performing ReCom moves. Since we require the district sizes to stay constant, the approach of \cite{Akitaya_et_al} does not work and we must use a different strategy.

\begin{proof}
Let red, blue, and green be the district colors. We prove the result by giving a sequence of ReCom moves such that the end result is three districts which are each $1 \times n$ rows.  This will imply that the metagraph is connected.

First, we define a district (say that district is red) to be \emph{contiguous against the left edge} if: when we start at any red cell and move to the left, we eventually reach the left edge of the rectangle without ever leaving the red district.  Suppose that there is a district (without loss of generality, the red district) which is contiguous against the left edge of the rectangle.  In that case, we combine the green and blue districts. In the combined green-blue district, start by looking only at the top and bottom rows. Determine which of these two rows is longer (extends further to the left). Color the longer row green and the shorter row blue. Then, color the middle row from left to right, starting with the correct number of green cells and finishing with blue cells. \Cref{fig:left_contig1} shows examples of this ReCom move.

\begin{figure}[h]
    \centering
    \vspace{-15pt}
    \includegraphics[width = .3\textwidth]{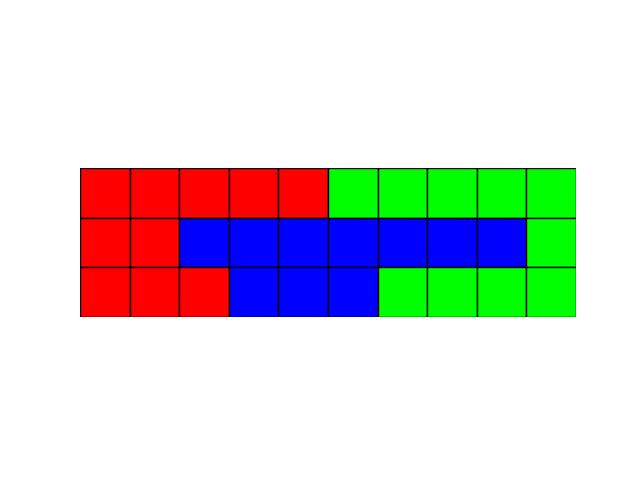}
    \includegraphics[width = .3\textwidth]{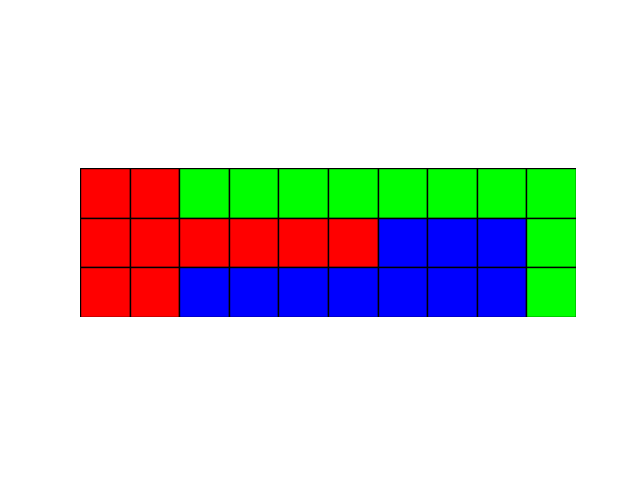}
    \includegraphics[width = .3\textwidth]{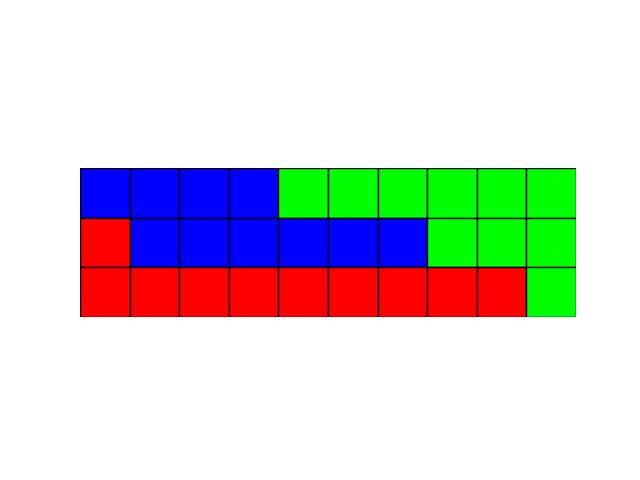}
    \vspace{-25pt}
    
    \includegraphics[width = .3\textwidth]{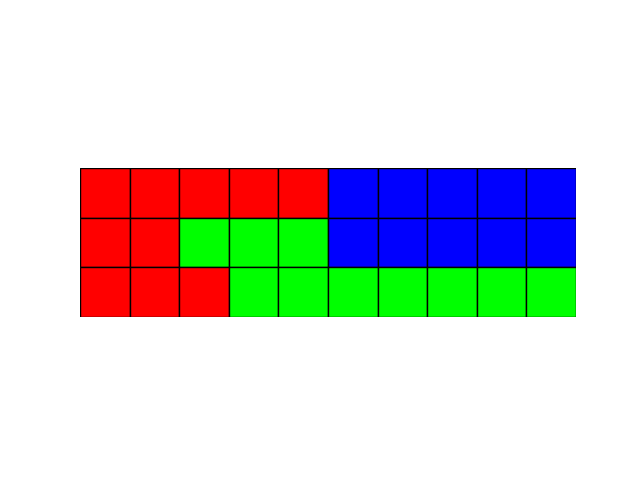}
    \includegraphics[width = .3\textwidth]{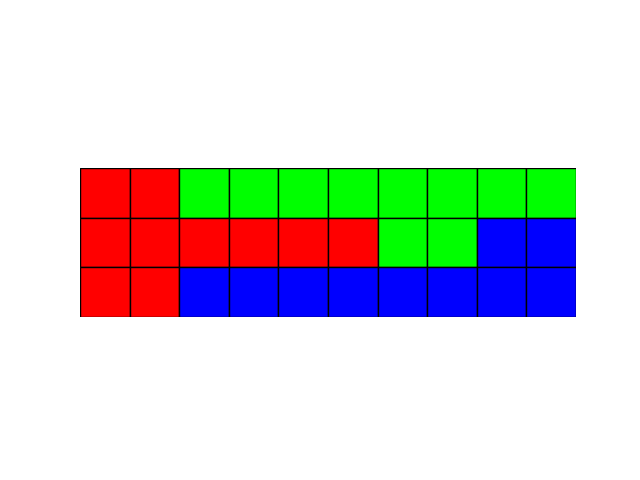}
    \includegraphics[width = .3\textwidth]{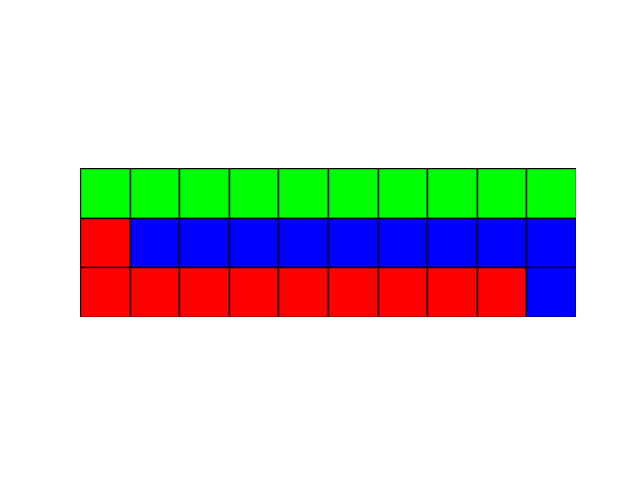}
    \caption{Top: Some arrangements of three districts in a $3 \times n$ grid where the red district is contiguous against the left edge. Bottom: The result of taking the top images, combining green and blue, coloring the longer of the top and bottom rows green and the shorter blue, and coloring the middle row from left to right starting with the correct number of green cells and finishing with the correct number of blue cells.}
\label{fig:left_contig1}
\end{figure}

Next, combine the red and green districts. The combined red-green district contains either the entire top row or the entire bottom row. Color this entire row green and the rest red. The green district is now a $1 \times n$ row. Finally, combine red and blue to make all three districts into $1 \times n$ rows. See \Cref{fig:left_contig2}.

\begin{figure}[h]
    \centering
    \vspace{-15pt}
    \includegraphics[width = .3\textwidth]{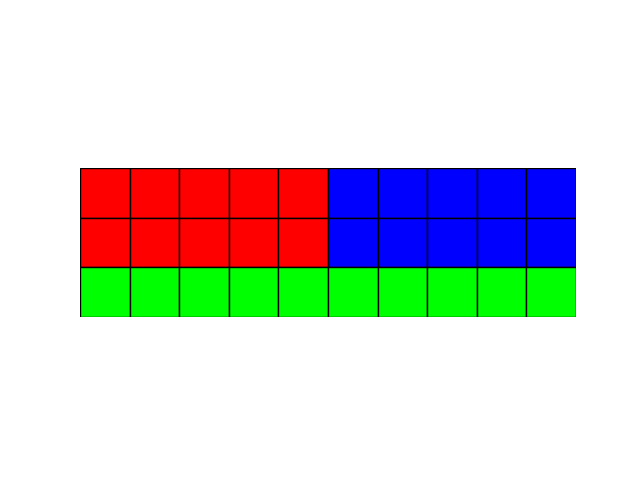}
    \includegraphics[width = .3\textwidth]{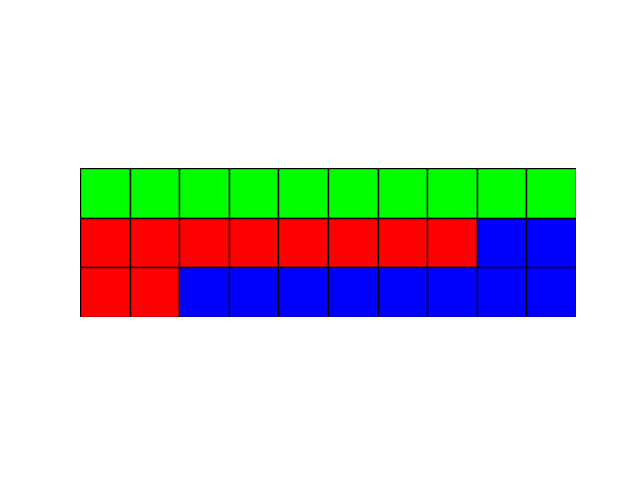}
    \includegraphics[width = .3\textwidth]{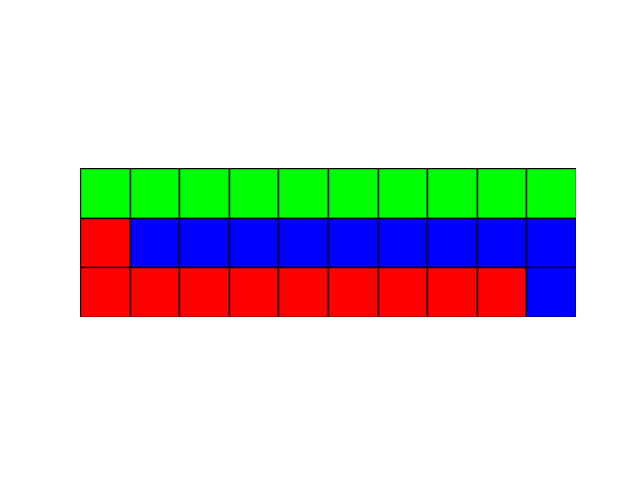}
    \vspace{-25pt}
    
    \includegraphics[width = .3\textwidth]{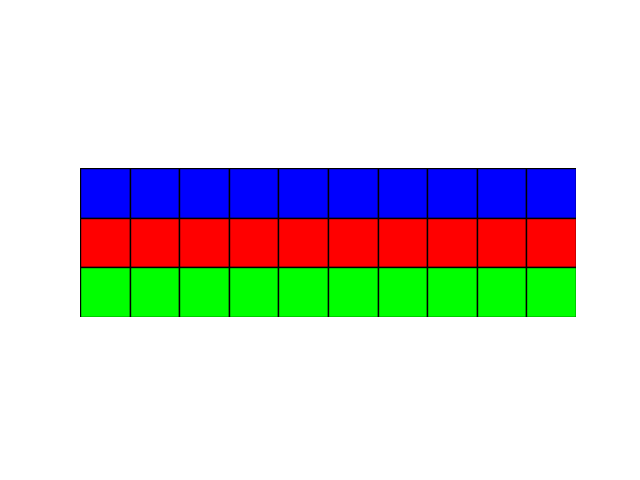}
    \includegraphics[width = .3\textwidth]{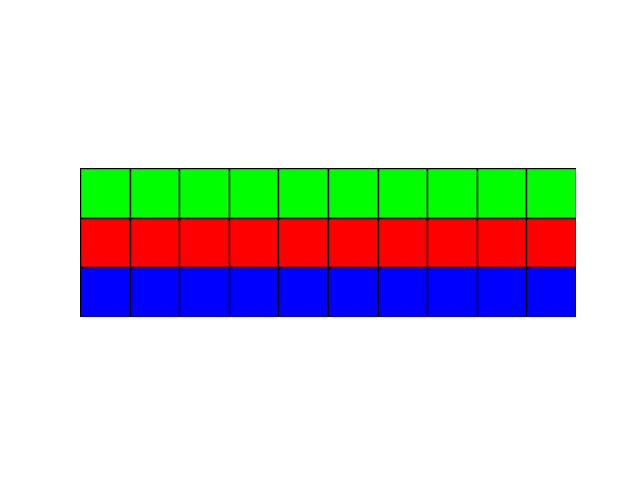}
    \includegraphics[width = .3\textwidth]{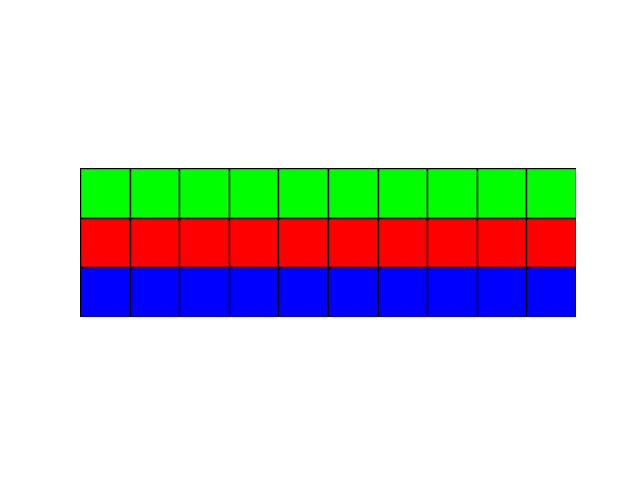}
    \caption{Top: The result of taking the bottom images in \Cref{fig:left_contig1}, combining red and green, and making an entire row green. Bottom: Combining red and blue to make three $1 \times n$ rows.}
\label{fig:left_contig2}
\end{figure}

If no district is contiguous against the left edge of the rectangle, we consider two cases. It will be useful to put a coordinate system on the $3 \times n$ grid: position $(j,k)$ means that $j$ is the vertical coordinate ($1 \leq j \leq 3$, going from top to bottom) and $k$ is the horizontal coordinate ($1 \leq k \leq n$, going from left to right).

\noindent\underline{Case 1:}  There is some district that, upon removal, disconnects the remaining grid into two connected components (the other two districts).

Without loss of generality, say that green is the district whose removal disconnects the graph.  If the green district consists of the entire middle row, then we are done. Otherwise, all three districts must intersect the middle row at some point. Assume without loss of generality that the leftmost non-green cell in the middle row is colored red. We show some sample images in the top row of \Cref{fig:disconnect1}. (For visual convenience, the bottom row of \Cref{fig:disconnect1} shows the ReCom moves that we will eventually make on these configurations.)

\begin{figure}[h]
    \centering
    \vspace{-15pt}
    \includegraphics[width = .3\textwidth]{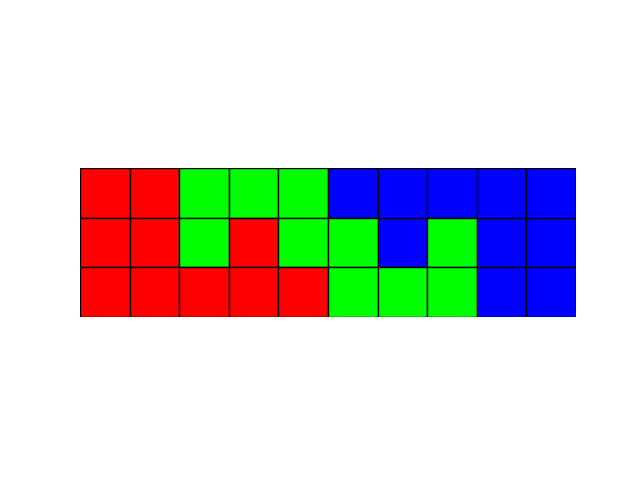}
    \includegraphics[width = .3\textwidth]{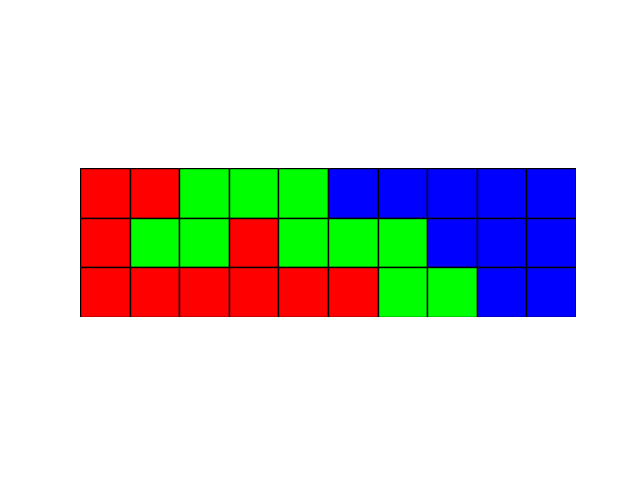}
    \includegraphics[width = .3\textwidth]{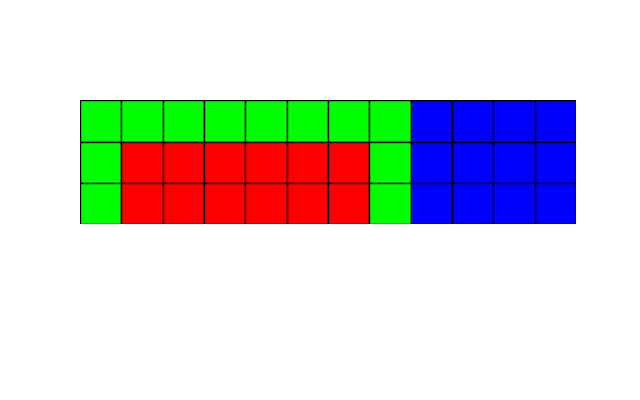}

    \vspace{-25pt}

    \includegraphics[width = .3\textwidth]{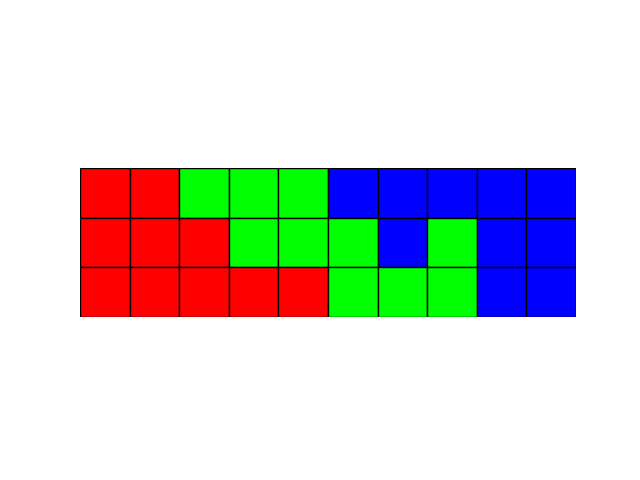}
    \includegraphics[width = .3\textwidth]{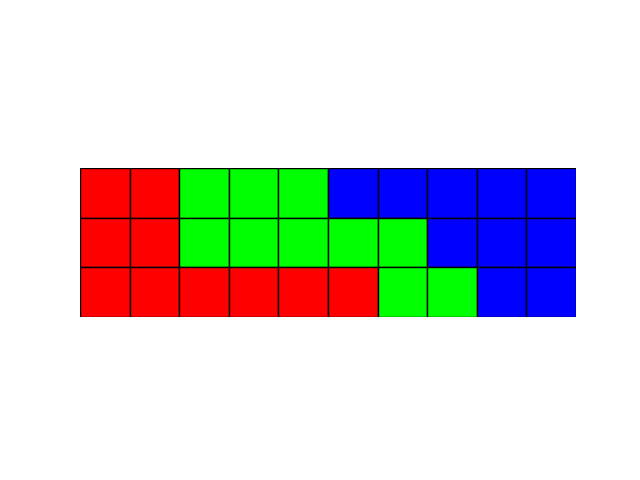}
    \includegraphics[width = .3\textwidth]{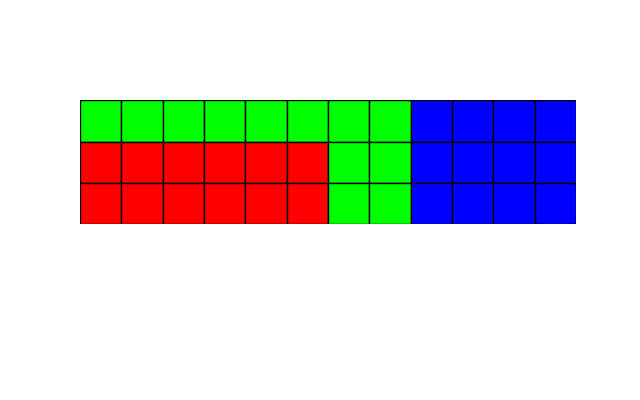}

    \caption{Top: Sample arrangements of three districts in a $3 \times n$ grid, where removal of the green district disconnects the dual graph and the leftmost non-green cell in the middle row is colored red. Bottom: The result of taking the top images, combining the red and green districts, and shifting the red cells to the left of each row.}
\label{fig:disconnect1}
\end{figure}

Observe that all three sample images in the top row of \Cref{fig:disconnect1} satisfy the following two properties:
\begin{enumerate}
    \item In each row individually, all the red cells are to the left of all the blue cells.
    \item The entire red district lies to the left of any column with a blue cell in the middle row. The entire blue district lies to the right of any column with a red cell in the middle row.
\end{enumerate}
We now prove that these two properties must hold in general.

We start with Property 2. Recall that both the red and blue districts must intersect the middle row. Let $(2,j)$ be any red cell in the middle row, and let $(2,k)$ be any blue cell in the middle row. Since the red and blue districts do not touch, column $j$ has no blue cells. This implies that the entire blue district is on either the left side or the right side of column $j$. Similarly, column $k$ has no red cells, so the entire red district is on either the left side or the right side of column $k$.

Let $(2,a)$ be the leftmost non-green cell in the middle row, which is red by assumption. Since cells $(2,1)$ through $(2,a-1)$ are all green, it must be true that $a<k$. Thus, by the previous paragraph, the entire red district is to the left of column $k$. In particular, we have that $j<k$, which implies that the entire blue district is to the right of column $j$. This proves Property 2.

Property 1 for the middle row is clearly implied by Property 2. We now prove Property 1 for the top row. Suppose for contradiction that there are a blue cell $(1,c)$ and a red cell $(1,d)$ with $c<d$. We know that $a<c$. Since neither $(1,a)$ nor $(1,d)$ is blue, the blue district must intersect the middle row at some cell $(2,b)$ where $a<b<d$. But then the entire red district lies to the left of column $b$, contradicting our assumption that $(1,d)$ is red. This proves Property 1 for the top row. The proof for the bottom row is the same.

With Properties 1 and 2 in hand, we are ready to make a ReCom move. Combine the red and green districts. Keep the number of red cells in each row the same, but shift them as far to the left as possible. Because of the two properties, the new red district is contiguous against the left edge of the rectangle (and thereby connected). We argue that the new green district is also connected. Indeed, starting from any newly colored green cell, one can move to the right until reaching a part of the green district that was unaffected by the ReCom move. The bottom row of \Cref{fig:disconnect1} shows the results of this ReCom move for the sample images.

Now that the red district is contiguous against the left edge of the rectangle, we have already shown how to perform more ReCom moves to make the districts into three $1 \times n$ rows.

\noindent\underline{Case 2:}  No matter which district we remove, the remaining districting grid is connected.

We know that some side (either column $1$ or column $n$) has two contiguous cells of the same color.  (If not, then  we must already have three districts which are $1 \times n$ strips.)  We can additionally assume, by flipping and renaming colors if needed, that the red district occupies cell $(2,1)$ and at least one of $(1,1)$ and $(3,1)$. Two sample configurations are shown in the top row of \Cref{fig:connect1}. As in \Cref{fig:disconnect1}, the lower rows of \Cref{fig:connect1} show the ReCom moves that will eventually be made.

\begin{figure}[h]
    \centering
\includegraphics[width = .35\textwidth]{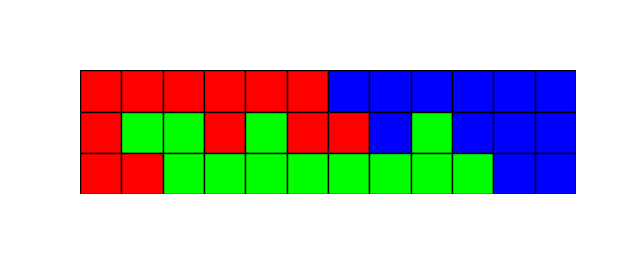} \qquad
\includegraphics[width = .35\textwidth]{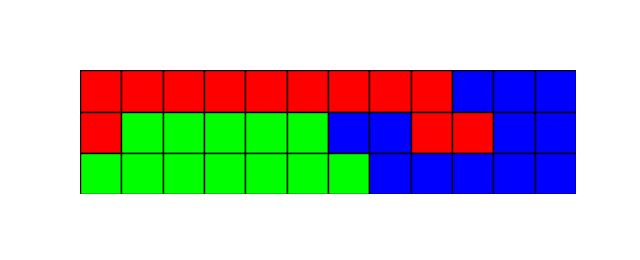}

\includegraphics[width = .35\textwidth]{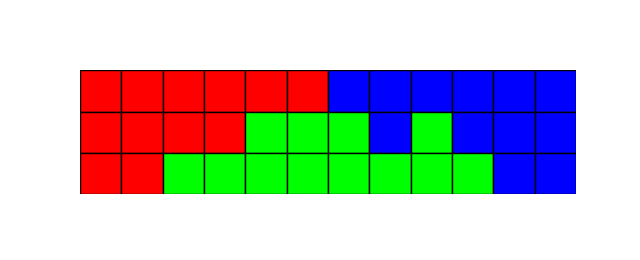} \qquad
\includegraphics[width = .35\textwidth]{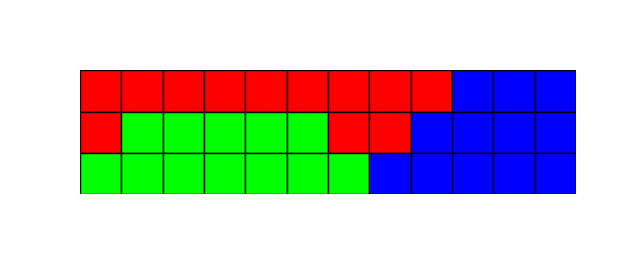}

\qquad \qquad \qquad \qquad \qquad \qquad \qquad \quad \ 
\includegraphics[width = .35\textwidth]{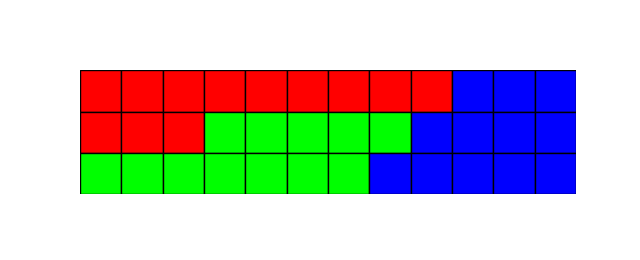}

\vspace{4pt}

\caption{Top row: Sample arrangements where removal of any district leaves a connected dual graph. The red district occupies cell $(2,1)$ and at least one other cell in the first column. \\[6pt]
Second row left: The result of taking the left configuration in the top row, combining the red and green districts, and shifting the red cells in row 2 to the left of the row. \\[6pt]
Second row right: The result of taking the right configuration in the top row, combining red and blue districts, and shifting the blue cells in row 2 to the right of the row. \\[6pt]
Third row right: The result of taking the right configuration in the second row, combining red and green districts, and shifting the red cells in row 2 to the left of the row.}
\label{fig:connect1}
\end{figure}

Our goal is to perform ReCom moves so that the red district is contiguous against the left edge of the rectangle. Since the graph remains connected when the red district is removed, all the red cells in row 1 must be contiguous against the left edge, and all the red cells in row 3 must be contiguous against the left edge. If the same is true in row 2, then our goal is accomplished, and we already know how to perform ReCom moves to make three $1 \times n$ districts.

Suppose that the red cells in row 2 are not contiguous against the left edge. Without loss of generality, assume that the leftmost non-red cell in row 2 is colored green. Let $(2,c)$ be the rightmost red cell in row 2, and let $(2,b)$ be the red cell with the property that all cells $(2,x)$ with $b \leq x \leq c$ are red, but cell $(2,b-1)$ is not red. In the top left image of \Cref{fig:connect1} we have $b=6$ and $c=7$, while in the top right image we have $b=9$ and $c=10$.

Since $(2,b-1)$ is not red, one of $(1,b)$ and $(3,b)$ must be red. The other cannot be red, or else removal of the red district would disconnect the graph. Without loss of generality, we may assume that $(1,b)$ is red and $(3,b)$ is not. It follows that all the cells in row 1 from $(1,1)$ to $(1,b)$ are red.

Consider the cells in row 2 from $(2,1)$ to $(2,b-1)$. Either there are only red and green cells here, or there are cells of all three colors.

\noindent\underline{Case 2a:} The cells in row 2 from $(2,1)$ to $(2,b-1)$ are only red and green.

This case is illustrated by the top left image in \Cref{fig:connect1}.

We perform a ReCom move by combining red and green. In row 1 and row 3, leave all cell colors unchanged. In row 2, shift all the red cells to the left of the row, and color the rest green. As a result, the mixed red-green cells from $(2,1)$ to $(2,c)$ are replaced by a red segment that goes to the right from a starting point of $(2,1)$ and a green segment that goes to the left from a starting point of $(2,c)$. The left image in the second row of \Cref{fig:connect1} shows the result of this ReCom move when applied to the top left image. As well, the transition from the second row on the right to the third row on the right in \Cref{fig:connect1} shows another instance of this ReCom move.

The new red district satisfies the desired property that all three rows are contiguous against the left edge. We must check that the new green district is connected. If $(3,c)$ is green, as in the images on the left side of \Cref{fig:connect1}, then the newly colored green cells are all reachable from $(3,c)$ by going up to $(2,c)$ and turning left. The cells that remained green are reachable from $(3,c)$ by the same green path that existed before the ReCom move. If $(3,c)$ is blue, as in the images on the right side of \Cref{fig:connect1}, then in the original configuration before the ReCom move was made, the green district was boxed in by the red cells from $(1,1)$ to $(1,b)$ and $(2,b)$ to $(2,c)$ along with the blue cell at $(3,c)$. Thus, the new green district post-ReCom move consists of a contiguous segment of cells along row 3, $(3,u)$ through $(3,v)$ for some $u \leq v \leq c-1$, and a contiguous segment of cells along row 2, $(2,w)$ through $(2,c)$ for some $w \leq c$. The horizontal coordinates of these segments must overlap because there are $n$ green cells in total and $c \leq n-1$. This ensures that the new green district is connected. In conclusion, we have performed a legal ReCom move so that all three rows of the red district are contiguous against the left edge of the rectangle.

\noindent\underline{Case 2b:} The cells in row 2 from $(2,1)$ to $(2,b-1)$ include cells of all three colors.

This case is illustrated by the top right image in \Cref{fig:connect1}.

In the second row from $(2,1)$ to $(2,b-1)$, there cannot be a blue cell to the left of a green cell, because then the blue district would be boxed in by both red and green and it would not have enough cells. Also, cell $(3,b)$ must be blue for the same reason, along with all cells to its right in row 3. Let $(2,a)$ be the leftmost blue cell in row 2. Row 2 is entirely red-blue to the right of this point.

We perform a ReCom move by combining red and blue. In row 1 and row 3, leave all cell colors unchanged. In row 2, shift all the blue cells to the right of the row, and color the rest red. As a result, the mixed red-blue cells from $(2,a)$ to $(2,c)$ are replaced by a red segment on the left and a blue segment on the right. Both new districts are connected because the new red cells can be reached from the red cells in row 1 and the new blue cells can be reached from the blue cells in row 3. The right image in the second row of \Cref{fig:connect1} shows the result of this ReCom move when applied to the top right image.

After applying the ReCom move, row 2 consists of a red-green segment on the left side and a blue segment on the right side. This is exactly the scenario of Case 2a. Thus, we repeat Case 2a by combining red and green, leaving all cell colors unchanged in rows 1 and 3, and shifting the red cells in row 2 to the left of the row. The result is illustrated by the right image in the third row of \Cref{fig:connect1}. We have reached a configuration where the red district is contiguous against the left edge of the rectangle.

In both Cases 2a and 2b, we have made the red district contiguous against the left edge of the rectangle. From there, we have already shown how to get three $1 \times n$ rows using more ReCom moves. This completes Case 2 and finishes the proof of \Cref{thm:3_by_n_rectangle}.
\end{proof}

\section{Conclusions and remaining questions}\label{sec:conclusions}

In \Cref{thm:main_result_triangular_lattice}, we have shown that, for any map whose dual graph is an \ctriangle  of side length $n>4$, the dimer tiling metagraph is connected.  In other words, for a family of nice districting maps whose districts are size 2, the Markov chain on this map using ReCom moves is irreducible.  

This result is meaningful because it adds to the short list of families of redistricting metagraphs which are known to be connected.  We note that Cannon's example of a family of connected redistricting metagraphs \cite{cannon_irreducibility} required 70+ pages of proof (though that was greatly shortened by Akitaya et al's proof \cite{Akitaya_et_al}); ours required 17+ days of compute time using 32 cores and 1 TB of memory.  

This result is also meaningful in that it has a direct application: many states require that their state senate districting maps be formed by pairing state house districts together.  Suppose a state has such a requirement.  Then \Cref{thm:main_result_triangular_lattice} states that, if a state house districting map has a dual graph which happens to be a dimer-tilable triangular region of side length $n>4$, then any possible state senate map can be constructed by starting with a single state senate map and performing a series of ReCom moves on that map.

\Cref{thm:main_result_triangular_lattice} applies only when the dual graph is in the shape of a large (full or clipped) triangle. It is natural to wonder whether the metagraph is still connected for other simple shapes such as parallelograms and hexagons. We expect so. The crucial step in the proof of \Cref{thm:main_result_triangular_lattice} is \Cref{lemma:inductive-gadget}, which requires the dual graph to have a straight boundary rather than a jagged boundary. (Indeed, \Cref{fig:herringbone} demonstrates that the metagraph may no longer be connected when the boundary of the dual graph is jagged.) We think that the metagraph should be connected whenever the dual graph is a simple shape with straight-line boundaries. Proving this would probably require additional extensive computation.

Of course, there are many other maps whose metagraph connectivity is unknown.  A reasonable next question would be:  what about a triangular subset of the triangular lattice, where all districts have size 3?  Is the corresponding metagraph connected?  We believe the answer is yes, though we may need to use significantly more compute power than even the massive amount we currently have access to.  

While \Cref{thm:main_result_triangular_lattice} shows that nice-enough dual graphs result in connected metagraphs, \Cref{thm:locked-6} reminds us that locked configurations still exist, even in nice graphs (nearly triangular graphs with degree no more than 8).  Our constructions for \Cref{thm:locked-6} are notable in that large portions of half of the districts look like the ``compact'' shape required by most states.  Specifically, in the limit, 1/4 of the nodes in our construction are in portions of the district that look ``compact.''  This is in contrast to the results in \cite{JTF_tilings}, where all locked tilings had non-compact shapes.  We are curious as to whether this ratio of 1/4 could be increased with a different configuration.

Our interpretation of \Cref{thm:locked-6} is that locked configurations can exist given ground rules that are sufficiently flexible. In \cite{JTF_tilings}, Tucker-Foltz needed the flexibility to assign population 1 to some nodes and population 2 to others in order to construct locked configurations. Similarly, \Cref{thm:locked-6} adds flexibility to the settings of \cite{Akitaya_et_al, cannon_irreducibility} by increasing the number of districts from three to six and, more importantly, by allowing the dual graph to be a bounded degree triangulation. With this extra freedom, the triangulation can be tuned just right for the construction of locked configurations.

We have found that this precise ``tuning'' is necessary for \Cref{thm:locked-6} to work.\footnote{See Step 3 in the proof of \Cref{thm:locked-6} for details.} This raises a natural question.  If we put some probability distribution on the space of possible dual graphs, can we say that locked configurations exist only with low probability (tending to zero as the number of nodes increases)? More speculatively, what can be said about the probability that the metagraph is connected?

Finally, we note that our \Cref{thm:3_by_n_rectangle} is the first positive result on the grid graph where the number of districts is fixed and the population of each district is not allowed to deviate.  It avoids the issues of intertwined districts by limiting ourselves to a narrow dual graph in which winding is impossible.  While we (and anecdotally many others) believe that the $3n \times 3n$ grid graph on 3 districts likely has a connected metagraph (after all, 2/3 of the map can change in one ReCom move!), we do not yet have any insights as to how one might deal with three districts winding around each other in the square lattice.

\bmhead{Acknowledgements}
We thank the anonymous referees for their insightful comments and suggestions, which improved the presentation of our results and helped us put them in appropriate context.  We also thank Lev Rappaport, who provided firsthand feedback on the accessibility of our images for colorblind readers.

\bibliography{Metagraph}

\begin{appendices}
\appendix 

\section{Invariants}\label{sec:invariants}

Chen et al explored the application of invariants in the context of the metagraph of redistricting maps \cite{ExploringMetagraphs}.  They explored these invariants specifically for the $6 \times 6$ grid graph, which was already known to have a disconnected metagraph \cite{JTF_tilings}.  These invariants are values which can be calculated for each tiling, and all tilings within a connected component of the metagraph must have the same corresponding set of values (hence the term ``invariants''; they are invariant under a ReCom move).  It is not guaranteed that two different connected components must have different invariant values, but if one can show that two different tilings have different invariant values, this automatically implies that they are in different connected components of the metagraph.

Here we define the invariants in question, and explore what they are capable and incapable of showing about tilings of the triangular lattice.  We use notation similar to the notation in \cite{ExploringMetagraphs}.

\subsection{Definitions for tiling invariants}

The following definition can be made for any 
set of tiles $\T$ on the hexagonal grid:

\begin{Definition}
    Let $\T$ be a set of tiles in the hexagonal grid in the plane, and let $R$ be a ring.  We define $R(\T)$ to be the $R$-module of formal linear combinations of tiles in $\T$, with coefficients in $R$.  
\end{Definition}

For our purposes, we will focus on dimer tiles (tiles of size two) and the cases where the ring in question is either $\Z$ or $\Z_2$.  There are three dimer tiles in the hexagonal grid; we label the northwest/southeast tiles $a$, the north/south tiles $b$, and the northeast/southwest tiles $c$, as in \Cref{fig:dominos}.   

\begin{figure}[h]
    \centering
    \includegraphics[width=0.5\linewidth]{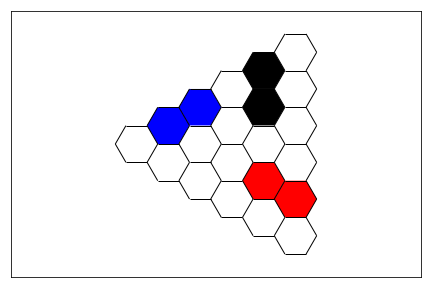}
    \caption{The red dimer is labeled $a$, black $b$, and blue $c$.}
    \label{fig:dominos}
\end{figure}

There are three different kinds of ReCom moves that can be made for the triangular lattice in this dimer setting; they can be visualized in \Cref{fig:recom_moves}.  

\begin{figure}[h]
    \centering
    \includegraphics[width=0.3\linewidth]{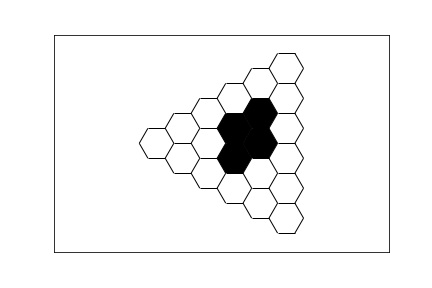}
    \includegraphics[width=0.3\linewidth]{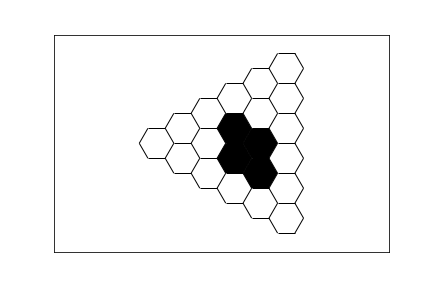}
    \includegraphics[width=0.3\linewidth]{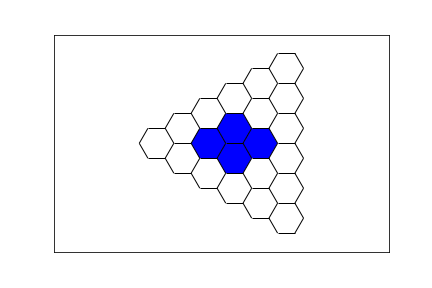}

    \includegraphics[width=0.3\linewidth]{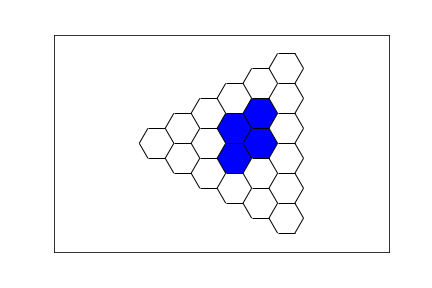}
    \includegraphics[width=0.3\linewidth]{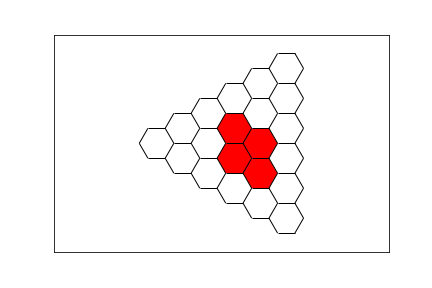}
    \includegraphics[width=0.3\linewidth]{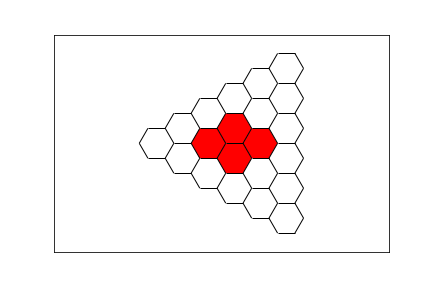}
    \caption{Each pair of tiles in the above images can be changed through a ReCom move to the pair of tiles below it.  The tiles are color-coded as in \Cref{fig:dominos}.}
    \label{fig:recom_moves}
\end{figure}

Again, the following definitions can be applied to any set of tiles $\T$.

\begin{Definition}
    Consider tiles $t_1$, $t_2$, $t_3$, and $t_4$, which are not required to be distinct.  Suppose tiles $t_1$ and $t_2$ can be adjacent to each other in the hexagonal grid in a way that their union is a connected region which can be re-split into tiles $t_3$ and $t_4$.  That is, tiles $t_1$ and $t_2$ can be combined and split into $t_3$ and $t_4$ (and vice versa) through a ReCom move.  Then we say that 
    \begin{equation*}
        t_1+t_2 - t_3 - t_4
    \end{equation*}
    is a ReCom combination of tiles or ReCom vector.  
\end{Definition}

Note that if $R$ is a ring which contains 1 and $-1$, then each ReCom combination of tiles is in the $R$-module $R(\T)$.  Finally, we have the following definition:

\begin{Definition}
    For a ring $R$ which contains $1$ and $-1$, we define $\I_{R}(\T)$ to be the set of vectors over $R$ whose dot product with every ReCom vector in $R(\T)$ is $0$.  In other words, suppose that the ReCom combination of tiles is written as $\sum_{t \in \T} \alpha(t)t$.  Then
    \begin{equation*}
        \I_{R}(\T) = \left\{f: \T \to R \Biggm\vert \sum_{t \in \T} f(t)\alpha(t) = 0  \; \text{ for all ReCom vectors } \; \sum_{t \in \T} \alpha(t)t\right\}
    \end{equation*}
\end{Definition}

Now suppose that $T$ and $T'$ are tilings using tiles in $\T$ of the same region $\mathcal{R}$ in the hexagonal grid, and that $T$ and $T'$ are in the same connected component of the tiling metagraph.  Let $v$ be the vector indexed by $\T$ whose value at index $t \in \T$ is the number of tiles of type $t$ in $T$, and let $v'$ be the corresponding vector for $T'$.   Every vector $u \in \I_{R}(\T)$ must satisfy
\begin{equation*}
    u \cdot v = u \cdot v'
\end{equation*}

Thus, we call these vectors $u \in \I_R(\T)$ \textbf{recombination invariants} (following the terminology of \cite{ExploringMetagraphs}).  Note that if two different indicator vectors of tilings of $\mathcal{R}$ using tiles in $\T$ have different values when dotted with some $u \in \I_R(\T)$, they \emph{must} be in different connected components of the tiling metagraph.  However, having the same such value when dotted with every $u \in \I_R(\T)$ does not guarantee that they are in the same connected component.

\subsection{Prior work on tiling invariants on redistricting metagraphs}

Chen et al \cite{ExploringMetagraphs} considered the recombination invariants $\I_{\mathbb{R}}(\T)$ when the underlying graph is the square grid, the region being tiled is the $6 \times 6$ square, and the set of tiles $\T$ are all trominos (connected tiles of size 3).  They showed that the number of NW ``chair'' trominos minus the number of SE chairs is invariant, as is the number of NE chairs minus the number of SW chairs.  They consider only the ring $\Z$ in their ring modules.  Since $\Z \subset \mathbb{R}$, the set of ReCom invariants can be found as a subset of the nullspace of the matrix whose rows are the vectors corresponding to ReCom combinations of tiles.  They show that the two invariants described above (NW chairs minus SE chairs, and NE chairs minus SW chairs) form a basis of that nullspace.  

Although there are \emph{many} different connected components of the tilings of the $6 \times 6$ grid with trominos (see \cite{JTF_tilings}), these invariants can\emph{not} distinguish between all of them.  For example, there are four connected components of size 384, each with different invariant values: $(3, 0)$, $(0, 3)$, $(-3, 0)$, and $(0, -3)$.  At the same time, of the 8 connected components of size 235: 2 components have invariant values $(3,0)$, 2 components have invariant values $(0,3)$, 2 components have invariant values $(-3,0)$, and 2 components have invariant values $(0,-3)$. Thus, there are different metagraph components with the same invariant values.

\subsection{Dimer tilings of the hexagonal grid}

The three unique dimers (tiles of size 2) in the hexagonal grid are shown in \Cref{fig:dominos}; let $\T$ be the set of those tiles.  We note $\I_\Z(\T)$ simply consists of multiples of the vector $(1, 1, 1)$, indicating that the total number of tiles in each tiling must be the same under a recombination move.  This doesn't tell us much, since the number of tiles must be the same whether the tilings are connected by a ReCom move or not.  

The basis of $\I_{\Z_2}(\T)$ is the vectors $(1, 0, 0)$, $(0, 1, 0)$, and $(0, 0, 1)$, indicating that the parity of the number of vectors of each type in each tiling must remain the same under a ReCom move.  This tells us something more, though not anything unexpected, since each ReCom move takes pairs of tiles of one type to pairs of tiles of another type (see \Cref{fig:recom_moves}). In fact, this parity invariant is even less impressive than it may appear, because it turns out that every region $\mathcal{R}$ of the hexagonal grid has only one ``parity class'' (choice of parities for the three tile types) to which all dimer tilings of $\mathcal{R}$ must belong. In other words, just like the total number of tiles, the parities are invariant among all tilings of $\mathcal{R}$ and not just those in the same connected metagraph component. 

In an effort to get more interesting invariants, we considered the parity of the coordinates of the ``base hexagon'' for each of our dimer tiles (see \Cref{fig:domino_centers} to see the base hexagons of each of the three tile types).  For each tile type $x \in \{a, b, c\}$, we let $x_i$ denote:
\begin{align*}
    i=1 &: &&  \text{ the base tile has (odd, odd) coordinates } \\
    i=2 &: &&  \text{ the base tile has (odd, even) coordinates } \\
    i=3 &: &&  \text{ the base tile has (even, odd) coordinates } \\
    i=4 &: &&  \text{ the base tile has (even, even) coordinates } \\
\end{align*}

\begin{figure}
    \centering
    \includegraphics[width=0.35\linewidth]{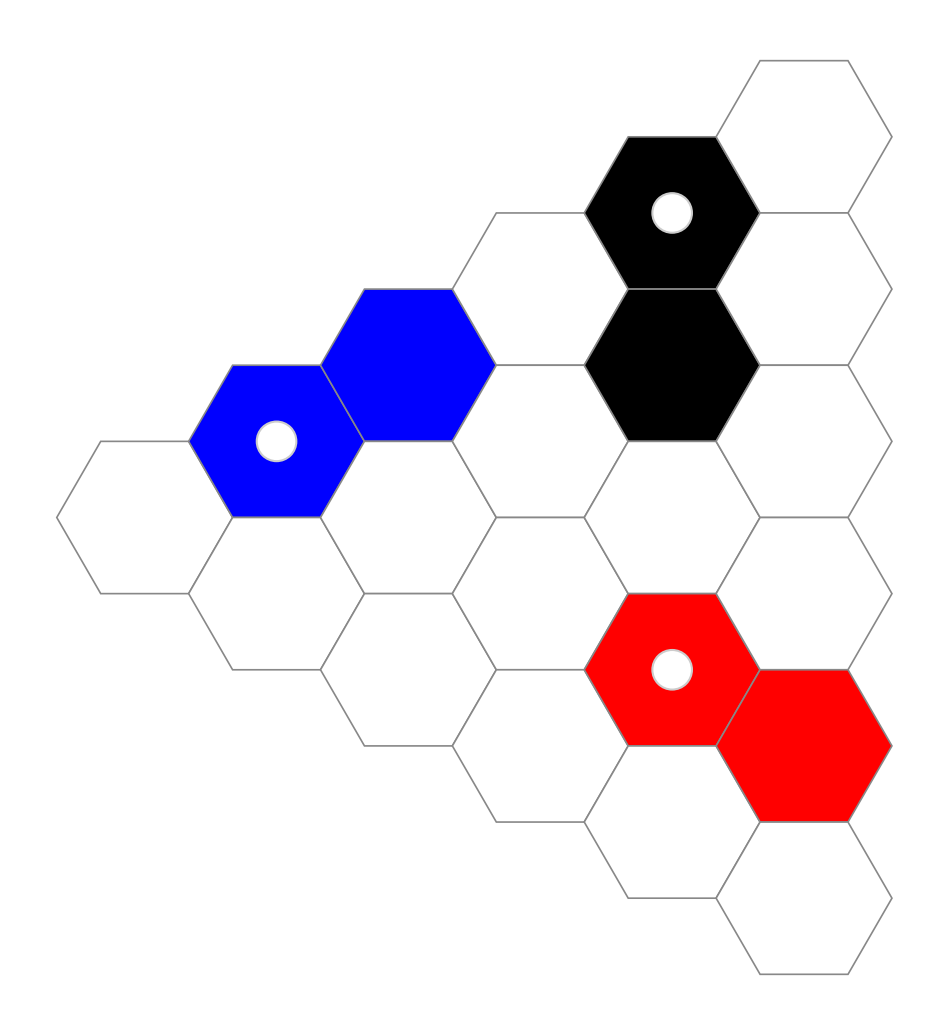}
    \caption{The white dot indicates the base hexagon for each of the three dimer tiles.}
    \label{fig:domino_centers}
\end{figure}

We denote by $\T_{\Z_2}$ the twelve tiles that result by distinguishing the parity of the base hexagon in tiles $a$, $b$, and $c$.   Using $\T_{\Z_2}$, our three ReCom combinations now expand into twelve combinations, depending on the parities of the combined tiles.  The space $\I_{\mathbb{R}}(\T_{\Z_2})$ now has dimension 4.  The 4 basis vectors for the nullspace of invariants can be described as follows:
\begin{enumerate}
    \item  The total number of tiles 
    \item  The number of $a$ tiles with even second coordinate minus the number of $a$ tiles with odd second coordinate
    \item The number of $b$ tiles whose coordinate sums are even minus the number of $b$ tiles whose coordinate sums are odd
    \item The number of $c$ tiles with even first coordinate minus the number of $c$ tiles with odd first coordinate
\end{enumerate}

We can also consider ReCom combinations over the field $\Z_2$ in order to find $\I_{\Z_2}(T_{\Z_2})$.  This results in a nullspace of dimension 5, with basis vectors of invariants described as follows:
\begin{enumerate}
    \item The parity of all tiles of type $a$
    \item The parity of all tiles of type $b$
    \item The parity of all tiles of type $c$
    \item The parity of the number of tiles of type $a$ with odd first coordinate plus the number of tiles of type $b$ with odd first coordinate
    \item The parity of the number of tiles of type $a$ with sum of coordinates odd, plus the parity of the number of tiles of type $b$ with sum of coordinates even, plus the parity of the number of tiles of type $c$ with odd first coordinate
\end{enumerate}

We conclude with a bit of speculation. Our overall goal is to answer the question: Given two tilings, is it possible to get from one to the other via ReCom moves? Suppose that we choose two tilings for which this is impossible. In that case, there ought to be some reason why it cannot be done. This could be a local obstruction in the metagraph, like if one of the tilings is a locked or partially locked configuration. Or, it could be a global obstruction in the metagraph, like if the two tilings have different invariant values. We might hope to prove a statement of the following form: Given two tilings, if neither one is locked or partially locked, and if they have the same invariant values, then they are in the same metagraph component. Indeed, this is precisely the form of the conjecture of \cite{moessner2001resonating} about the dimer tiling metagraph on a hexagonal grid with periodic boundary conditions. They conjecture that after removing the 12 locked configurations, the metagraph has four connected components given by parity invariants. As a cautionary note, the parity invariants of \cite{moessner2001resonating} do not seem to fit nicely into the framework of \cite{ExploringMetagraphs} that we have explored in this Appendix. This indicates that further development of the invariant theory is needed. 

\section{Algorithms used in computer search}
\label{sec:pseudocode}

\newcommand{\Shapes}{\mathcal{S}}
\newcommand{\Covers}{\mathcal{C}}
\newcommand{\Tilings}{\mathcal{T}}
\newcommand{\Moves}{\mathcal{M}}
\newcommand{\Part}{\mathcal{P}}
\newcommand{\set}[1]{\left\{#1\right\}}
\newcommand{\card}[1]{\left|#1\right|}
\newcommand{\true}{\textbf{true}}
\newcommand{\false}{\textbf{false}}

In this section, we present the algorithms we implemented for the computer search used in the proof of \Cref{thm:main_result_triangular_lattice}.  The core optimizations in our implementation are: pre-computing the ReCom moves for each pair of tile shapes, and parallelizing the search over potential covers of a region. For further optimizations in our implementation, we direct the readers to the source code, which is publicly available at \url{https://github.com/maemre/triforce}.
Throughout this section, $k$ refers to the tile size: although we use the algorithms only for tile size 2 in this work, they can work for larger tile sizes as well.

In our description of the algorithms, we distinguish \emph{tile shapes} which denote unique tile shapes independent of where a tile is placed, and \emph{tiles} which are placed on fixed coordinates in a region.

The core algorithm (\Cref{alg:parallel}) takes in:
\begin{itemize}
    \item a set of nodes $A$ that must be covered (the red nodes in \Cref{fig:inductive-gadget,fig:bottom-gadgets}),
    \item a set of nodes $B$ that \emph{may} appear in the cover (the union of red and blue nodes), and
    \item a partial tiling $\Part$ (the orange tiles) which we desire to exist in each connected component of each metagraph.
\end{itemize}
We define a \emph{cover of $A$} to be a set of nodes $R$ with $A \subseteq R \subseteq B$. \Cref{alg:parallel} checks that all covers $R$ of $A$ satisfy the following condition: In the metagraph of the tilings of $R$ by tiles of size $k$, each connected component contains a tiling that is an extension of $\Part$ (meaning a tiling that contains all the tiles in $\Part$).
The algorithm checks the metagraph for each cover in parallel. Each cover is checked by invoking \Cref{alg:check-one} for that cover.

\begin{algorithm}[htbp]
\caption{Check the whole set of covers in parallel}
\label{alg:parallel}
\begin{algorithmic}[1]
\Function{CheckAllCoversInParallel}{$A,B,k,\Part$}
    \Input $A$ is the region that must be tiled, $B$ is the set of nodes that can be tiled, $k$ is the tile size, $\Part$ is the desired partial tiling that must occur in each connected component of each metagraph.
    \State $\Covers \gets \set{A \cup D \mid D \subseteq B}$
    \State $\Shapes_k\gets$ pre-compute tile shapes of size $k$
    \State $\Moves_k\gets$ pre-compute potential ReCom moves for two tile shapes of size $k$
    \ForAll{$R \in \Covers$ \textbf{in parallel}}
        \If{$\neg \Call{CheckCover}{R,k,\Part,\Shapes_k,\Moves_k}$}
            \State \Return \false
        \EndIf
    \EndFor
    \State \Return \true
\EndFunction
\end{algorithmic}
\end{algorithm}

\Cref{alg:check-one} checks the metagraph condition described above by first building the set of all tilings of the region $R$ by tiles of size $k$, and then building the set of tilings of $R$ that can be reached by starting from an extension of the partial tiling $\Part$ and then applying a sequence of ReCom moves. The desired condition is that the latter set should be equal to the former, and not a proper subset.

In the algorithm, $V$ corresponds to the set of nodes of the metagraph (i.e., the set of all tilings of $R$), and $\mathit{covered}$ is the set of tilings reachable from an extension of $\Part$ by a sequence of ReCom moves.
The check at line~\ref{code:guard} identifies tilings we haven't seen yet ($T \notin covered$) that are an extension of $\Part$ ($\Part \subseteq T$). Given such a tiling $T$, the entire connected component of the metagraph that contains $T$ should be added to the set $\mathit{covered}$; this is accomplished in line~\ref{code:mark}.
\Cref{alg:check-one} does not use $k$ explicitly, but we pass $k$ along in the pseudocode for clarity as it determines the values of the arguments $\Shapes_k$ and $\Moves_k$.

The helper function $\textsc{EnumerateCompleteTilings}(R, \Shapes)$ iteratively builds the tilings for a region $R$ using the set of tile shapes $\Shapes$: it starts with an empty tiling, it generates new tilings by adding a tile to one of the existing partial tilings, and it repeats this process until reaching a fixpoint (i.e., until it stops generating tilings that are not previously generated).
At the end, it keeps only the complete tilings of $R$.

The helper function $\textsc{ReachableTilings}(T, V, \Moves_k)$ takes a tiling $T$, a set of possible tilings $V$, and a set of ReCom moves $\Moves_k$. It returns the subset of $V$ consisting of those tilings that are reachable from $T$ via a sequence of ReCom moves in $\Moves_k$, that is, the connected metagraph component containing $T$.  It computes this via depth-first-search in the metagraph induced by $V$ and $\Moves_k$.

\begin{algorithm}[H]
\caption{Check one cover by implicit metagraph traversal}
\label{alg:check-one}
\begin{algorithmic}[1]
\Function{CheckCover}{$R,k,\Part,\Shapes_k,\Moves_k$}
    \Input
    $R$ is the region, $k$ is the tile size, $\Part$ is the desired partial tiling, $\Shapes_k$ and $\Moves_k$ are the pre-computed tile shapes and ReCom moves for the given tile size.
    \State $V\gets\Call{EnumerateCompleteTilings}{R,\Shapes_k}$
    \If{$V=\varnothing$}
        \State \Return $\true$
    \EndIf
    \State $\mathit{covered}\gets\varnothing$
    \ForAll{$T\in V$}
        \If{$T\notin\mathit{covered}$ and
                $\Part \subseteq T$}
        \label{code:guard}
            \State $\mathit{covered}\gets\mathit{covered}\cup \Call{ReachableTilings}{T,V,\Moves_k}$
            \label{code:mark}
        \EndIf
    \EndFor
    \State \Return $\mathit{covered}=V$
\EndFunction
\end{algorithmic}
\end{algorithm}

\end{appendices}

\end{document}